 \documentclass[onefignum,onetabnum]{siamonline250211}

\usepackage{lipsum}
\usepackage{amsfonts}
\usepackage{graphicx}
\usepackage{epstopdf}
\usepackage{algorithmic}
\usepackage{mathtools}
\usepackage{tikz}
\usepackage{booktabs}
\usepackage{tabularx}
\ifpdf
  \DeclareGraphicsExtensions{.eps,.pdf,.png,.jpg}
\else
  \DeclareGraphicsExtensions{.eps}
\fi

\usepackage{mathtools, hyperref, color}
\usepackage[shortlabels]{enumitem}
\usepackage{graphicx}
\usepackage{adjustbox}
\usepackage{xspace}
\usepackage{cleveref}
\usepackage{multirow}
\usepackage{booktabs}
\usepackage{array}
\usepackage{tabularx}
\usepackage{tikz,tikz-cd,tikz-3dplot}
\usetikzlibrary{matrix, arrows, positioning, calc, spy}
\tikzset{
  mirror scope/.is family,
  mirror scope/angle/.store in=\mirrorangle,
  mirror scope/center/.store in=\mirrorcenter,
  mirror setup/.code={\tikzset{mirror scope/.cd,#1}},
  mirror scope/.style={mirror setup={#1},spy scope={
      rectangle,lens={rotate=\mirrorangle,yscale=-1,rotate=-1*\mirrorangle},size=80cm}},
}
\usepackage{float}

\usepackage{enumitem}
\setlist[enumerate]{leftmargin=.5in}
\setlist[itemize]{leftmargin=.5in}

\newsiamremark{remark}{Remark}
\newsiamremark{hypothesis}{Hypothesis}
\crefname{hypothesis}{Hypothesis}{Hypotheses}
\newsiamthm{claim}{Claim}
\newsiamremark{fact}{Fact}
\crefname{fact}{Fact}{Facts}
\newsiamremark{example}{Example}

\headers{Squared Linear Models}{H.\ Friedman, B.\ Sturmfels and M.\ Wiesmann}

\title{Squared Linear Models
}

\author{Hannah Friedman\thanks{UC Berkeley, CA 94720-3840, United States 
  (\email{hannahfriedman@berkeley.edu}).}
\and Bernd Sturmfels\thanks{MPI MiS, 04103 Leipzig, Germany (\email{bernd@mis.mpg.de}).}
\and Maximilian Wiesmann\thanks{CSBD, Dresden, 01307 Germany 
  (\email{wiesmann@pks.mpg.de}).}}

\definecolor{cb-yellow}{RGB}{221,170,51}
\definecolor{cb-red} {RGB}{187,85,102}
\definecolor{cb-teal}{RGB}{0,153,136}
\definecolor{cb-blue} {RGB}{0,68,136}
\definecolor{cb-green}{RGB}{17,119,51}
\definecolor{cb-purple} {RGB}{170,68,153}
\definecolor{cb-palegrey} {RGB}{221,221,221}

\newcommand{\PP}{\mathbb{P}}
\newcommand{\R}{\mathbb{R}}
\newcommand{\QQ}{\mathbb{Q}}
\newcommand{\CC}{\mathbb{C}}
\newcommand{\ZZ}{\mathbb{Z}}
\newcommand{\NN}{\mathbb{N}}

\newcommand{\cA}{\mathcal{A}}

\newcommand\mirror[1][]{\spy[overlay,#1] on (\mirrorcenter) in node at (\mirrorcenter)}

\newcolumntype{P}[1]{>{\centering\arraybackslash}p{#1}}
\ifpdf
\hypersetup{
  pdftitle={Squared Linear Models},
  pdfauthor={H.\ Friedman, B.\ Sturmfels, and M.\ Wiesmann}
}
\fi

\begin{document}

\maketitle

% REQUIRED
\begin{abstract}
We study statistical models that are parametrized by squares of linear forms. All critical points of the likelihood function are real and positive. There is one critical point in each region of the projective hyperplane arrangement defined by the linear forms. We examine the ideal and singular locus of the model, and we give a determinantal presentation for  its likelihood correspondence. We characterize tropical degenerations of the MLE, we describe 
the log-normal polytopes, and we explore connections to determinantal point processes.
\end{abstract}

% REQUIRED

% REQUIRED

 \section{Introduction}

We consider a discrete statistical model on $n$ states
which is given by linear forms
$\ell_1,\ldots,\ell_n$ in real variables $x_1,\ldots,x_d$,
where $n > d > 1$. In this model,
the probability of observing the $i$-th state is
 \begin{equation}
 \label{eq:para}
  p_i(x) \,\, = \,\, \frac{\ell_i^2(x)}{\sum_{j=1}^n \ell_j^2(x)} \qquad {\rm for} \,\, i = 1,2,\ldots,n. 
  \end{equation}
Here $x = (x_1,\ldots,x_d)$ are the model parameters. The denominator
is the partition function.
We write $X$ for the projective variety in $\PP^{n-1} = \PP^{n-1}_{\CC}$
parametrized by (\ref{eq:para}) and $I_X$ for its prime ideal in
$\CC[p_1,p_2,\ldots,p_n]$. Our aim is to explain the
likelihood geometry \cite{HS} of~$X$.

\begin{example}[$d=3,n=4$]
\label{ex:d=3n=4}
Consider the model $X$ defined by four lines in  $\PP^2$, namely
$$ \ell_1 = x_1, \,\,
\ell_2 = x_2, \,\,
\ell_3 = x_3 \,\,\,\, {\rm and} \,\,\,\,
\ell_4 = x_1+x_2+x_3. $$
Then $X$ is a {\em Steiner surface} in $\PP^3$,
by \cite[Example 3.5]{DGK}, with ideal $I_X$ generated by the quartic
\begin{equation}
\label{eq:quartic}
\begin{matrix}
 p_1^4+p_2^4+p_3^4+p_4^4
 \,\,+\,\,6\, (p_1^2 p_2^2+p_1^2 p_3^2+p_1^2 p_4^2+p_2^2 p_3^2+p_2^2 p_4^2 + p_3^2 p_4^2) \\
-\,4 (p_1^3 p_2+p_1^3 p_3+ \cdots+p_3 p_4^3)
\,+\,4 (p_1^2 p_2 p_3+p_1^2 p_2 p_4
+\cdots +p_2 p_3 p_4^2) \, -\,40 p_1 p_2 p_3 p_4.
\end{matrix}
\end{equation}
The surface $X_\R$ is singular along three lines. Its smooth part has
 $16 = 4 + 3 \cdot 4$ connected components in the tetrahedron $\Delta_3 = \PP^3_{> 0}$.
 These components are  formed by the $7 = 4 + 3$ regions of the
 line arrangement  $\mathcal{A} = \{\ell_1,\ell_2,\ell_3,\ell_4\}$
 in the plane $\PP^2_\R$.
The likelihood function
$$  \quad p_1(x)^{s_1} p_2(x)^{s_2} p_3(x)^{s_3} p_4(x)^{s_4},
\qquad \hbox{for given data $s_1,s_2,s_3,s_4 \in \NN$,} $$
is positive on $\PP_\R^2 \backslash \mathcal{A}$.
It has seven complex critical points,  all real, one in each region.
\end{example}

In general, 
the likelihood function of a 
squared linear model (\ref{eq:para})
can have arbitrarily many
local maxima inside the probability simplex. 
We next present an example to show this.

\begin{example}[The braid arrangement] \label{ex:braid}
  Fix $c = d+1$, $n = \binom{c}{2}$, and let $X$ be the model
  \begin{equation}
\label{eq:braidmodel}
 p_{ij}(x) \,\, = \,\,\frac{(x_i-x_j)^2}{c \, (\sum_{k=1}^c x_k^2)\, -\, (\sum_{k=1}^c x_k)^2}
\qquad {\rm for}\,\,\, 1 \leq i < j \leq c. 
\end{equation}
The ideal $I_X$ is generated by the $2 \times 2$ minors of the symmetric
$d \times d$ matrix with
diagonal entries $\,2 p_{ic}\,$ and
off-diagonal entries $\,p_{ic} + p_{jc} - p_{ij}\,$
for $1 \leq i,j \leq d$.
Thus $X$ is a {\em Veronese variety}. It has dimension $c-2$ and degree $2^{c-2}$
in $\PP^{\binom{c}{2}-1}$. The likelihood function has
 $c\,!/2$ complex critical points.
 All of them are real and positive. There is one such point in each region
of the braid arrangement in $\PP_\R^{c-1}$.
These regions are indexed by the $c\,!$ permutations of
$\{1,2,\ldots,c\}$, modulo reversal involution.
We view $X$ as a subvariety of the {\em squared Grassmannian} ${\rm sGr}(2,c+1)$,
which was studied in \cite{DFRS, Fri}.
Both models have the same ML degree.
\end{example}

We now discuss the structure of this paper, and we summarize our main results.
Given the model (\ref{eq:para}), we fix data $s_1,\ldots,s_n \in \R_{>0}$.
We are interested in the log-likelihood function 
\begin{align}\label{eq:loglikelihoodfn}
  x \,\,\mapsto \,\,  s_1 \log\,\ell^2_1(x) + \cdots + s_n \log \,\ell^2_n(x) - (s_1 + \cdots + s_n)\log\bigl(\ell_1^2(x) + \cdots + \ell_n^2(x)\bigr).
\end{align}
In Section~\ref{sec2} we prove that all
complex critical points of~(\ref{eq:loglikelihoodfn}) are real, and there is precisely one critical point in
each connected component of $\,\PP^{d-1}_\R \backslash \mathcal{A}$.
Hence the ML degree equals the number of such regions, which is
$\sum_{i=0}^{d-1} \binom{n-1}{i}$ when the hyperplane arrangement $\mathcal{A}$ is generic;
see, e.g.,~(\ref{eq:genfcn}). This result is Theorem \ref{thm:two}. It rests on recent work of  Reinke and Wang~\cite{RW}.

In Section \ref{sec3} we study the variety $X \subset \PP^{n-1}$,
focussing on generic arrangements $\mathcal{A}$.
For large $n$, the ideal $I_X$ is generated by linear forms and the
$2 \times 2$ minors of a symmetric $d \times d$ matrix (Proposition \ref{prop:veronese}).
For small $n$, the variety $X$ is a projection of the
Veronese variety $\nu_2(\PP^{d-1})$. 
The ideal $I_X$ is complicated, and tough computations are needed.
Proposition \ref{prop:ideal_computations} summarizes what we currently know.
The singular locus of $X$ is determined in Theorem \ref{thm:singularities}.

Section \ref{sec4} is devoted to the likelihood correspondence.
This is the variety in $\PP^{n-1}\times \PP^{d-1}$ which parametrizes
pairs $(s,x)$ where $x$ is a critical point of  the log-likelihood function 
(\ref{eq:loglikelihoodfn}) for the data $s$.
We study its bihomogeneous prime ideal using the methods in
 \cite[Section~2]{KSSW}. Theorem \ref{thm:likely}  expresses the likelihood ideal for
generic arrangements as a determinantal ideal.  The ideal
is minimally generated by
$\binom{n}{d-2}$ polynomials of bidegree $(1,n-d+2)$ in $\CC[s,x]$.

In Section \ref{sec5} we study likelihood degenerations \cite{ABFKST}
when $\mathcal{A}$ is generic.
Theorem  \ref{thm:happydegeneration} shows that the critical points
are distinct even when $s$ is a unit vector~$e_i$.
We then turn to tropical maximum likelihood estimation (MLE) for squared linear models.
Tropical MLE for linear models appeared in \cite[Section 7]{ABFKST},
and it was developed for toric models in \cite{BDH}. 
Corollary   \ref{cor:tropgeneric} is an analogue to \cite[Theorem 7.1]{ABFKST},
but now for all regions of the arrangement $\mathcal{A}$,
not just bounded regions.

In Section \ref{sec6} we investigate
log-normal polytopes and
log-Voronoi cells, as defined in \cite{AH}. For each distribution $p$ in the  model $X$, the 
former parametrizes data $s$ such that $p$ is critical for~(\ref{eq:loglikelihoodfn}),
while the latter parametrizes data $s$
such that $p$ is the MLE.
These convex bodies agree for linear models \cite{Ale}.
They disagree for squared linear models, even when both are polytopes (Example  \ref{ex:nongeneric-uniform2}).
The log-normal polytopes  are characterized in Theorem \ref{thm:chamber}.

In Section \ref{sec7}, we show how squared linear models arise in applications. 
Namely, we prove that when the hyperplane arrangement $\mathcal A$ is a discriminantal arrangement, then the corresponding model is a linear projection determinantal point process (Proposition~\ref{prop:discriminantal}).

\smallskip

Software and data for computational results in this paper are
 presented on a Zenodo page, available at
\url{https://doi.org/10.5281/zenodo.15997158}.

\section{Real and Positive}
\label{sec2}

We will show that all complex critical points $x$ of (\ref{eq:loglikelihoodfn})
are real. This implies that
$\ell^2_1(x), \ldots, \ell^2_n(x)$ are positive, and every critical point of the log-likelihood function yields a probability distribution
in $\Delta_{n-1}$.
A similar argument was used in \cite{Fri} to show that the log-likelihood function on the squared Grassmannian has only positive critical points. 
By \cite[Theorem 1.7]{HS},
the number of critical points of (\ref{eq:loglikelihoodfn}) is the signed Euler characteristic
of the very affine variety underlying our statistical model.
This variety is the complement in $\PP^{d-1}$
of the $n$ hyperplanes $V(\ell_1),\ldots,V(\ell_n)$ and the quadric
$V(q)$ defined by $q = \ell_1^2 + \cdots + \ell_n^2$.
Its Euler characteristic is an upper bound on the number of real critical points. 
One  argues that there is at least one critical point per region of 
$\PP^{d - 1}_\R \backslash \!\cup_{i=1}^n \! V(\ell_i)$.
This gives a lower bound on the number of real critical points. 
The punchline is that the two bounds agree.

\begin{theorem} \label{thm:two}
For generic data $s \in \R_{>0}^n$, all complex critical points of the log-likelihood function \eqref{eq:loglikelihoodfn}
are real, and there is one critical point in each region of $\,\PP_{\R}^{d-1} \backslash \mathcal A$. 
Hence every critical point on the implicit model $X \subset \PP^{n-1}$ is positive and is a local maximum.
\end{theorem}

As a consequence, any local optimization method, e.g. gradient ascent, used to find the MLE is likely to fail, since there is an abundance of spurious local maxima. 
Instead, computing the MLE via homotopy continuation might be more appropriate.\par 

To prove Theorem~\ref{thm:two}, we apply \cite[Theorem 1.2]{RW}, which states that, if the partition function $q =\ell_1^2 + \cdots + \ell_n^2$ is replaced with a generic quadric $g$, then all critical points of \eqref{eq:loglikelihoodfn} are~real. 
By generic we mean that the  hypersurface $V(g)$ is in general position relative to the hyperplane arrangement $ \mathcal A = \{V(\ell_1), \ldots, V(\ell_n)\}$, i.e., no nonempty intersection $\bigcap_{i \in I} V(\ell_i)$ for $I \subseteq [n]$ is tangent to $V(g)$.
We will 
prove that the partition function $q$ satisfies this.

\begin{lemma}\label{lem:generic_quadric}
    The quadratic hypersurface defined by
     $q = \ell_1^2 + \cdots + \ell_n^2$ in $\PP^{d-1}$ is in general position with respect to the hyperplane arrangement 
     $\mathcal A$ that underlies the squared linear model.
     \end{lemma}

\begin{proof}
  We prove that, for any nonempty subset $I \subseteq [n]$, the intersection $ \bigcap_{i \in I} V(\ell_i)$
    is  not tangent  to the hypersurface $V(q)$.
  Let $A $ be the $n \times d$ matrix satisfying $Ax = (\ell_i(x))_{i \in [n]}$.
  We assume that $A$ has rank $d$, so the  arrangement $\mathcal A$ is essential. 
  Otherwise, $d - {\rm rank}(A)$ of the variables can be eliminated, resulting in an essential arrangement in $\PP^{\,{\rm rank}(A) - 1}$. 
    If $A_i$ denotes the $i$th row of $A$, then the gradient of the partition function is 
  \begin{align*}
    \nabla (q(x)) \,\,= \,\,2\ell_1(x) A_1 + \cdots + 2\ell_n(x) A_n     \,\, = \,\,2x^T A^T A.
  \end{align*}
  We proceed by induction on  the cardinality $|I|$.
    We first claim that $V(\ell_i)$ is not tangent to $V(q)$ for $i = 1, \ldots, n$.
  This is equivalent to proving that the $2 \times d$ matrix
    \begin{align*}
    C \,\,= \,\,\begin{pmatrix}
      x^T A^T A\\
      A_i
    \end{pmatrix}
  \end{align*}
    has rank $2$ at every point $x \in V(q) \cap V(\ell_i)$. We now show the contrapositive.
    
  Suppose there exists  $\alpha \in \CC^*$ such that $\alpha x^T A^T A = A_i$.
  Since $A_i$ is real and nonzero and $A^TA$ is a real invertible matrix, $\alpha x$ is also real and nonzero.
  Since $q$ is a sum of squares and $\alpha \neq 0$, we have $q(\alpha x) = \alpha^2 q(x) \neq 0$.
  Thus the matrix $C$ has rank 2 for any $x \in V(q) \cap V(\ell_i)$.

We now consider an index set $I$ of cardinality $r$. After relabeling, 
$I = \{1,2,\ldots,r\}$, and we assume that  $\bigcap_{i=1}^r V(\ell_i)$  is non-empty.
We must show that this intersection is not  tangent to $V(q)$.
If it were tangent, there would exist $\lambda_1, \ldots, \lambda_r$ such that $\lambda_1 A_1 + \cdots + \lambda_r A_r =  x^TA^TA$.
  Letting ${\rm proj}_{\ell_i}(v)$ denote the orthogonal projection of $v$ onto the hyperplane $V(\ell_i)$, this implies
  \begin{align}\label{eq:tangentprojection}
  0 + \lambda_2 {\rm proj}_{\ell_1}A_2 + \cdots + \lambda_r {\rm proj}_{\ell_1}A_r 
  \,\,=\,\,  {\rm proj}_{\ell_1}(x^TA^TA).
  \end{align}
   We now view the intersections $V(\ell_2) \cap V(\ell_1), \ldots, \,V(\ell_n) \cap V(\ell_1), \,V(q) \cap V(\ell_1)$ as subvarieties of $V(\ell_1) \cong \PP^{d - 2}$.
  In that ambient space, the normal vector of $V(q) \cap V(\ell_1)$ is ${\rm proj}_{\ell_1}(x^TA^TA)$ at $x$.
  Similarly,  the normal vector of $V(\ell_i) \cap V(\ell_1)$ is ${\rm proj}_{\ell_1}(A_i)$.
  The condition  \eqref{eq:tangentprojection} means that the intersection 
  $\bigcap_{i=2}^r(V(\ell_i) \cap V(\ell_1))$  is tangent to 
  $V(q) \cap V(\ell_1) = V(\ell_2^2 + \cdots + \ell_n^2) \cap V(\ell_1)$
  inside $ \PP^{d - 2}$.
  This is a contradiction to our induction hypothesis. We conclude that the $\ell_i$ are in general position relative to the quadric $q$ in the complex projective space $\PP^{d-1}$.
\end{proof}

\begin{remark}
    Lemma~\ref{lem:generic_quadric} can fail if the coefficients of the $\ell_1, \ldots, \ell_n$  are complex.
    For example, if $d = 2, n = 3$ and $\ell_1=  \sqrt{-1} \cdot x,\,\, \ell_2 = y,\,\, \ell_3 = x+ y$, then the quadric factors as $q = \ell_1^2 + \ell_2^2 + \ell_3^2 = 2y(x + y)$. The intersection $V(q) \cap V(\ell_2) = V(\ell_2)$ is not transverse.
    \end{remark}

We will now prove the main result in this section.

\begin{proof}[Proof of  Theorem \ref{thm:two}]
   By Lemma~\ref{lem:generic_quadric},  the quadric $V(\ell_1^2 + \cdots + \ell_n^2)$ is in general position 
   with respect to the arrangement $\mathcal A$. As in the previous proof, we assume that $\mathcal{A}$ is essential, 
   since we can always reduce the number of variables to obtain an essential arrangement.
     
    We work in an affine chart where $\ell_1(x)= 1$; the quadric is still general with respect to $V(\ell_2), \ldots, V(\ell_n)$ and the arrangement $\mathcal A' = \{V(\ell_2), \ldots, V(\ell_n)\}$ is essential. Indeed, since $n > d$, we may assume that the normal of $V(\ell_1)$ is in the span of the normals of $V(\ell_2), \ldots, V(\ell_n)$.

On our affine chart $\{\ell_1(x) = 1\} \cong \CC^{d-1}$,
we consider the affine log-likelihood function
    \begin{align}\label{eq:likelihood_chart}
     2s_2 \log\, \ell_2(x) + \cdots + 2s_n \log\, \ell_n(x) \,-\, (s_1 + \cdots + s_n) \log(1 + \ell_2(x)^2 + \cdots + \ell_{n}(x)^2).
    \end{align}
    We do not lose any critical points when moving to the affine chart, since all critical points
    of (\ref{eq:loglikelihoodfn}) have $\ell_1(x) \neq 0$.
    By \cite[Theorem 1.2]{RW}, the affine log-likelihood function \eqref{eq:likelihood_chart} has only real critical points, and there is  exactly one critical point per region of 
    the affine hyperplane arrangement $\R^{d-1} \backslash \mathcal A'$.
    Hence the log-likelihood function \eqref{eq:loglikelihoodfn} has only real critical points.
    It has exactly one critical point per region of the 
    projective hyperplane arrangement $\PP_{\R}^{d-1} \backslash \mathcal A$. 

The implicit model $X$ is a subvariety of $\PP^{n-1}$.
    For a critical point $p^*$ in $X$, there is a critical point $x^* \in\PP^{d-1}$ of 
     \eqref{eq:loglikelihoodfn}  such that $p^*_i = \ell^2_i(x^*)$ for all $i$.
              Since $x^*$ is real, all complex critical points $p^*$
    on $X$ are real with positive coordinates.
 The likelihood function $\prod_{i=1}^n \ell_i(x)^{2s_i}$ is positive on $\PP_{\R}^{d-1} \backslash \mathcal A$ and zero on $\mathcal A$, so every critical point is a local maximum. 
 \end{proof}

The number of regions of a hyperplane arrangement
$\mathcal{A}$ in real projective space $\PP_\R^{d-1}$
is computed by means of  the {\em characteristic polynomial}
$\chi_{\mathcal{A}}(t)$ of the matroid represented by $\mathcal{A}$.
See \cite[Section 2]{RW} for definitions
and \cite{BEK} for state of the art on computations.
Zaslavsky's Theorem states that
the number of regions of the central arrangement in $\R^d$ equals $| \chi_\mathcal{A}(-1)|$.

\begin{corollary}\label{cor:MLdegree}
The ML degree of the squared linear model (\ref{eq:para}) is equal to
$| \chi_\mathcal{A}(-1)|/2$.
\end{corollary}

\begin{example}
The  braid arrangement in Example \ref{ex:braid} has the characteristic polynomial 
$$\chi_\mathcal{A}(t) \,\,=\,\, (t-1)(t-2)\,\cdots \,(t-(c-1)).$$
Hence the ML degree of the squared linear model (\ref{eq:braidmodel}) is equal to
$\,|\chi_\mathcal{A}(-1)|/2 \,=\, c\,!/2$.
\end{example}

\begin{remark}
The variety $X$ has degree $2^{d-1}$  in $\PP^{n-1}$ provided
$\mathcal{A}$ contains $d+1$ hyperplanes
in general position. In this case,
the map $\PP^{d-1} \rightarrow X $  is birational,
and the ML degree of the parametric model agrees with that
of the implicit model.   To prove this, consider a generic
fiber $\{(\pm x_1: \pm x_2 : \cdots : \pm x_d )\}$ of the coordinatewise
squaring map $\PP^{d-1} \rightarrow \PP^{d-1}$. These $2^{d-1}$ points have distinct images in $\PP^d$
if we tag on the extra coordinate $ (\pm x_1\, +\, \pm x_2 \,+\, \cdots \,+\, \pm x_n)^2$.

A birational inverse from $X$ to $\PP^{d-1}$ can be found using Gr\"obner bases, but
its coordinates are generally complicated.
For instance, here is such an inversion formula for Example~\ref{ex:d=3n=4}:
$$ \frac{x_1}{x_3} \,\, = \,\,\, \frac{
p_1^2+p_2^2+p_3^2+p_4^2
-2 p_1 p_2+6 p_1 p_3-2 p_1 p_4-2 p_2 p_3-2 p_2 p_4-2 p_3 p_4}{
4 p_3 (-p_1+p_2-p_3+p_4)}. $$

If the matroid of $\mathcal{A}$ is not  connected, then
the map $\PP^{d-1} \rightarrow X$ is not birational, and
the degree of $X$ in $\PP^{n-1}$ is strictly smaller than $2^{d-1}$. For instance, if we
replace Example \ref{ex:d=3n=4} by
$$ \ell_1 = x_1, \,\,
\ell_2 = x_2, \,\,
\ell_3 = x_3 \,\,\,\, {\rm and} \,\,\,\,
\ell_4 = x_1+x_2, $$
then $X$ is a quadratic cone in $\PP^3$, and its parametrization
$\PP^{2} \rightarrow X $  is two-to-one.
\end{remark}

Let $X_{>0}$ denote the intersection of the variety $X$ with the
open probability simplex $\Delta_{n-1} = \PP^{n-1}_{> 0}$.
In statistics, $X_{>0}$ is the natural domain of the log-likelihood function.

\begin{proposition} \label{prop:components}
If the variety $X$ is smooth and its parametrization
$\PP^{d-1} \rightarrow X $  is birational, then
its positive part $X_{>0}$ is a manifold with precisely $|\chi_\mathcal{A}(-1)|/2$ connected 
components. This holds for generic arrangements $\mathcal{A}$ when
$n \geq 2d-1$, but it does not hold for $n=d+1 \geq 4$.
\end{proposition}

\begin{proof}
Here we make a forward reference to the results on singularities in
Theorem \ref{thm:singularities}. Since $X$ is smooth, the 
parametrization is one-to-one outside the $n$ hyperplanes
$V(\ell_i)$ of the arrangement $\mathcal{A}$. Hence the parametrization
restricts to a homeomorphism between $\PP^{d-1}_\R \backslash \mathcal{A}$
and $X_{>0}$. The inverse is given by identifying the
correctly signed square root for each $p_i(x) = \ell_i(x)^2$.
The result hence follows from Zaslavsky's Theorem.
The failure for $n=d+1$ was seen for the Steiner surface in
Example \ref{ex:d=3n=4}. Here the singular locus has codimension one,
and the map attaches pairs of different regions of $\PP^{d-1}_\R \backslash \mathcal{A}$ 
along the singular locus.
\end{proof}

\section{Implicit Model Description}
\label{sec3}

In this section, we study the implicit representation of 
a squared linear model  in $\PP^{n-1}$. Our aim is to
describe its homogeneous prime ideal. We first focus on the generic case.
  Let $X_{d,n}$ denote the model for $n$ generic linear forms in $d$ variables,
  and write $I_{d,n}$ for its prime ideal. Thus $I_{d,n}$ is the ideal of
  algebraic relations among the squares of $n$ generic linear forms in $x_1,\ldots,x_d$.
  That ideal is easy to describe in the regime when the number $n$ is large enough relative to $d$.

  Fix $N = \binom{d+1}{2}$.
  Let $L$ be the $n \times N$ matrix whose $i$th row consists of the coefficients of $\ell_i^2(x)$.
  The columns of $L$ are indexed by quadratic monomials in $x_1,\ldots,x_d$. 
  Let $[L\,\, p]$ denote the $n \times (N + 1)$ matrix obtained from $L$ by appending a column of variables $p = (p_1\,\cdots\,p_n)^T$. 
  To construct the quadratic generators of $I_{d,n}$, we define $L'$ to be
  the $N \times N$ matrix formed by the first $N$ rows of $L$.
  The matrix $L'$ is invertible because the linear forms $\ell_1,\ldots,\ell_N$
  are generic, so their squares span the space of all quadratic forms.
  We  define the column vector
  \begin{equation}
    \label{eq:symmatrix}
    r \,\,=\,\, L'^{-1} \cdot (p_1\,\,p_2\,\,\cdots\,\,p_N)^T . 
  \end{equation}
  The entries of $r$ are indexed by the
  quadratic monomials in $x_1,\ldots,x_d$.
  These are naturally identified with the entries
  of a symmetric $d \times d$ matrix $\mathcal{R} = (r_{ij})$.
  Hence $\mathcal{R}$ is a  $d \times d$ matrix whose entries
  are linear forms in the first $N= \binom{d+1}{2}$ coordinates $p_1,\ldots,p_N$ of $\,\PP^{n-1}$.
  Given any matrix $M$ and $t \in \NN$, we write $I_t(M)$ for the ideal generated by the $t \times t$ minors of $M$. 
  
\begin{proposition} \label{prop:veronese} Let $n \geq N$.
  The prime ideal of $X_{d,n}$ equals
$\,    I_{d,n} = I_{N+1}([L \,\, p]) + I_2(\mathcal R)$.
This ideal is minimally generated by
  $n-\binom{d}{2}$  linear forms and
$  \frac{1}{12}(d+1) d^2 (d-1)$
  quadrics.
  The variety $X_{d,n}$
  is isomorphic to the quadratic
  Veronese embedding of $\,\PP^{d-1}$ into $\,\PP^{N-1}$.
  \end{proposition}
  
  \begin{proof}
    The $(N+1) \times (N+1)$ minors
    of  $[L\,\, p]$ give linear forms in $p$ that vanish on the variety $X_{d,n}$.
    Since $L$ has rank $N$, precisely $n-N$ of these linear forms are linearly independent.

By construction, the matrix $\mathcal{R}$ has rank one on the variety $X_{d,n}$.
This means that the $2 \times 2$ minors of $\mathcal{R}$ lie in the ideal $I_{d,n}$.
Together with the $n-N$ independent linear forms from $[L\,\, p]$ they
generate a prime ideal of the correct codimension.
Hence this ideal equals~$I_{d,n}$.

A symmetric $d \times d$ matrix has $\binom{\binom{d}{2}+1}{2}$ minors
of size two. These are linearly dependent for $d \geq 4$. Namely, they
obey  $\binom{d}{4}$
linear relations. Hence the space spanned by the
$2 \times 2$ minors of a generic symmetric $d \times d$ matrix has dimension
$  \frac{1}{12}(d+1) d^2 (d-1)$.
This is also the dimension of
the Schur module $S_{2,2}(\CC^d)$, so our count follows from
\cite[Proposition 6.10.4.1]{Landsberg_Tensors}.

  We claim that $X_{d,n}$ is isomorphic to $\nu_2(\PP^{d-1})$.
  Indeed, the coordinate ring of $X_{d,n}$~is
  $$
    \CC[p_1,\ldots,p_n]/I_{d,n} \,=\, \CC[p_1,\ldots,p_n]/(I_2(\mathcal R) + I_{N+1}([L\,\, p])) \,\cong\,
    \CC[p_1,\ldots,p_N]/I_2(\mathcal R).
$$
  The isomorphism follows from the fact that the maximal minors of $[L\,\, p]$ express $p_{N+1},
  \ldots,p_n$  as linear combinations of  $p_1, \ldots, p_N$. 
  Now, apply ${\rm Proj}$ to these isomorphic graded rings.
    \end{proof}

The braid arrangements in Example \ref{ex:braid} yield an infinite family of models
with $n=N$.
The first instance in this family of quadratic Veronese varieties is the plane curve
$\nu_2(\PP^1) \subset \PP^2$.

\begin{example}
Let $d =2,n=N=3$ and consider the linear forms $\{x,y,x + y\}$. We have
$$     L \,=\, L' \,=\, \begin{pmatrix}
        1 & 0 & 0\\
        0 & 0 & 1\\
        1 & 2 & 1
        \end{pmatrix}
\quad \hbox{and hence} \quad\,\,
        L'^{-1}p \,=\, \frac 1 2 \begin{pmatrix}
        \,\,2 & \,\,0 & 0 \,\\
        \! -1 & \!-1 & 1\,\\
        \,\,0 & \,\,2 & 0\,
        \end{pmatrix}  p      
        \,\,=\,\, \frac 1 2 \begin{pmatrix}
            2p_1 \\ p_3 \!-\! p_1 \!-\! p_2\\ 2p_2
        \end{pmatrix}\!.
$$
    The model $\,X_{2,3} \cong \nu_2(\PP^1)\,$ is the conic in the 
    plane $\PP^2$ defined by $\,4p_1p_2 = (p_3 - p_1 - p_2)^2$.
  \end{example}

\begin{example}
  Consider the seven lines in $\PP^2$ defined by $\{x, y, z, x+y+z, x+2y+3z, x+5y+7z, x+11y+13z\}$. 
  Here, $d=3, n=7$ and  $N = \binom{3+1}{2} = 6$. We consider the $7 \times 7$ matrix
  \[
      [L\,\, p] \,\,\,=\,\,\,  \begin{footnotesize} \begin{pmatrix}
        1 & 0 & 0 & 0 & 0 & 0 & p_1 \\
        0 & 0 & 1 & 0 & 0 & 0 & p_2 \\
        0 & 0 & 0 & 0 & 0 & 1 & p_3 \\ 
        1 & 2 & 1 & 2 & 2 & 1 & p_4 \\
        1 & 4 & 4 & 6 & 12 & 9 & p_5 \\
        1 & 10 & 25 & 14 & 70 & 49 & p_6 \\
        1 & 22 & 121 & 25 & 286 & 169 & p_7 \\
      \end{pmatrix}.\end{footnotesize}
  \]
  The first six columns are labeled by $(x^2, xy, y^2, xz, yz, z^2)$. The upper left $6 \times 6$ block 
  of $[L\,\,p]$ is the matrix $L'$. Its inverse yields the vector $\,r  = L'^{-1} p$, whose coordinates are 
  the entries~in
  $$ \mathcal{R} \,\, = \,\, \begin{pmatrix} r_1 & r_2 & r_4 \\
                                                 r_2 & r_3 & r_5 \\
                                          r_4 & r_5 & r_6 
  \end{pmatrix}.
  $$ 
The prime ideal $I_{d,n}$ of the model $X_{3,6} \subset \PP^5$ is generated by one linear form and six quadrics.
The linear form is the
determinant  of $[L \,\, p ]$, and the quadrics are the $2 \times 2$ minors of $\mathcal{R}$.
\end{example}

We now turn to the regime where $n$ is smaller than $N = \binom{d+1}{2}$.
Here the ideal generators are much more complicated than in Proposition \ref{prop:veronese}.
Not much is known about the ideals of  such projections in general.
The following result collects what we know in small dimensions.

\begin{proposition}
    \label{prop:ideal_computations}
    Table \ref{tab:deg_mingens} gives
    the degrees of the generators of $I_{d,n}$ for
    $d \leq 6 $ and $n < N$.
    \end{proposition}
\vspace{-0.25em}
\begin{table}[h]
\vspace{-0.3in}
    \centering
    \[ 
    \begin{array}[t]{c c}
         (d,n) & \text{Degrees} \\
         \hline
         (3,4) & 4^1 \\
         (3,5) & 3^7 \\
         \\
         (4,5) & 8^1 \\
         (4,6) & 4^1 5^{10} 6^2 \\
         (4,7) & 4^{45} \\
         (4,8) & 2^2 3^{20} 4^1 \\
         (4,9) & 2^{10} \\
    \end{array}
    \begin{array}[t]{|c c}
         (d,n) & \text{Degrees} \\
      \hline
         (5,6) & 16^1  \\
         {\bf(5,7)} & {\bf 9^{71}10^{35}11^1} \\
         (5,8) & 5^1 6^{98} \\
         (5,9) & 5^{286}  \\ 
         (5,10) & 3^{10} 4^{120}  \\
         (5,11) & 3^{76}  \\
         (5,12) & 2^8 3^{58} \\
         (5,13) & 2^{21} 3^2  \\
         (5,14) & 2^{35}  \\
    \end{array}
    \begin{array}[t]{|c c}
         (d,n) & \text{Degrees}\\
         \hline
         (6,7) & 32^1 \\
         (6,8) & \textcolor{gray}{15^1 16^{178} 17^{3708} 18^{11} } \\
         (6,9) & \textcolor{gray}{10^{1338} 11^{1668}}\\
         (6,10) & \textcolor{gray}{7^{694} 8^{1158}}\\
         (6, 11) & 6^{1820} \\
         (6,12) & 4^{78} 5^{429} \\
         (6, 13) & 4^{533} \\
    \end{array}
    \begin{array}[t]{|c c}
         (d,n) & \text{Degrees}\\
      \hline
         (6, 14) & 3^{98} \\
         (6, 15) & 3^{218} \\
         (6, 16) & 2^{10} 3^{194} \\
         (6, 17) & 2^{27} 3^{48} \\
         (6, 18) & 2^{45} \\
         (6, 19) & 2^{64} \\
         (6,20) & 2^{84} \\
    \end{array} \vspace{-0.2in}
    \]
    \caption{Degrees of the minimal generators (multiplicities in the exponents) for the prime ideals $I_{d,n}$.
    The {\color{gray} gray} entries refer not to minimal generators but to Gröbner basis elements.}
    \label{tab:deg_mingens}
\end{table}
\vspace{-1.25em}
\begin{proof}[Proof and Discussion]
Proposition \ref{prop:ideal_computations}
 was obtained by computer algebra, utilizing the Gr\"obner basis implementation in \texttt{Oscar.jl} \cite{oscar}.
Computations were performed over the rational numbers or finite fields.
We computed Gr\"obner bases and extracted
minimal generators. The cases $(6,8),(6,9),(6,10)$ are more challenging and will be completed in the future.

The ideal $I_{d,n}$ comprises relations among squares
      $\ell_i^2$ of generic linear forms $\ell_i$.
     If we replace these squares with generic quadrics
     then most entries of Table \ref{tab:deg_mingens} remain unchanged.
     The smallest exception is
     $(d,n) = (5,7)$.
      Here, generic quadrics yield different numbers of minimal generators.
The prime ideal $I_{5,7}$ has $107 $ minimal generators. There are
$71$ generators in degree $9$, and $35$ in degree $10$, and one in degree $11$.
By contrast, the prime ideal of relations among seven generic quadrics in
five variables requires $112$ minimal generators, namely
$70$ in degree $9$ and $42$ in degree $10$.
The $(5,7)$ entry would 
be ${\bf 9^{70} 10^{42}}$ for generic quadrics.
\end{proof}

What follows is a brief discussion of
the prime ideal $I$  for the squared linear model $X$ given by an arbitrary hyperplane arrangement
$\mathcal{A}$. The linear forms $\ell_1,\ell_2,\ldots, \ell_n$
are allowed to be special. The ideal $I$ was studied for small $d$ by
Dey, G\"orlach and Kaihnsa in \cite[Section~4.3]{DGK}.
Their analysis rests on the point configuration dual to the arrangement $\mathcal{A}$.
 Let $A_i = \nabla_x \ell_i(x)  \in \R^d$ be the vector of coefficients of $\ell_i$.
We view $A_1,\ldots,A_n$ as points in $\PP^{d-1}$.

The ideal $I$ depends on the  space of quadrics in $\PP^{d-1}$ that pass through the points $A_1,\ldots,A_n$.
If there are no such quadrics then we are in the generic situation of Proposition~\ref{prop:veronese}.
If the $A_i$ span a unique quadric that is irreducible then $I$ is generated
by $n-N-1$ linear forms together with the nonlinear equations of $X_{d,N-1}$.
For instance, if $d=3$ and $A_1,A_2,\ldots,A_n$ lie on a conic in $\PP^2$
then $I$ is generated by $n-5$ linear forms and $7$ cubics.
This is case (a) in \cite[Theorem 4.9 (ii)]{DGK}. If the $A_i$ lie on several
quadratic hypersurfaces then we are led to a case distinction 
as in \cite[Theorem 4.9 (iii)]{DGK}. See also \cite[Section 4.2]{DGK} for
 an interesting connection to the variety of
symmetric matrices with degenerate eigenvalues.

\smallskip

We now turn to the question
whether the generic model $X_{d,n}$ is smooth.
This happens when $n$ is large relative to $d$.
For instance, if $d=3$ then the model is
singular for $n=4$, by Example \ref{ex:d=3n=4},
but it is smooth for $n \geq 5$. This threshold is explained
by the following~theorem.

\begin{theorem} \label{thm:singularities} If $n > 2d-2$ then
 the generic squared linear model $X_{d,n}$ is smooth in $\PP^{n-1}$.
If $n \leq 2d-2$ then $X_{d,n}$ is singular, and the singular locus has dimension $2d -n-1$.
It is the image of all linear subspaces in $\PP^{d-1}$ on which the map $x \mapsto (\ell_1^2(x): \cdots : \ell_n^2(x))$ is not injective.
In particular, for $n=2d-2$ the singular locus of $\,X_{d,n}$ consists of $\frac{1}{2} \binom{2d-2}{d-1}$ lines.
\end{theorem}

\begin{proof}
  For $n \geq N = \binom{d+1}{2}$, the variety $X_{d,n}$ is smooth by Proposition~\ref{prop:veronese}.
      Next suppose $2d-2 < n < N$. Then $X_{d,n}$ is the image of $\nu_2(\PP^{d-1}) \subset \PP^{N-1}$ under a generic linear projection $\pi\colon \PP^{N-1} \dashrightarrow \PP^{n-1}$. The center $E$ of $\pi$ has dimension   $N-n-1$. 
      The secant variety $\sigma_2(\nu_2(\PP^{d-1}))$ has dimension $2d-2$, by \cite[Exercise 5.1.2.4]{Landsberg_Tensors}. 
      Therefore, since $2d-2 < n$, a general linear space $E$ does not meet the secant variety of 
      $\nu_2(\PP^{d-1})$.   Thus, by~\cite[Corollary 2.7]{shafarevich},
     $\pi$ defines an isomorphic embedding
    of $\PP^{d-1}$ into $\PP^{n-1}$, and smoothness is preserved.

Now, let $n \leq  2d-2$. Since $\mathcal{A}$ is generic, the 
rank of the Jacobian of our map is~constant on $\PP^{d-1} \backslash \mathcal{A}$.
All singularities of $X_{d,n}$ arise from 
pairs $\{x,x'\}$ of distinct points in $\PP^{d-1}$ which have the same image in $X_{d,n} \subset \PP^{n-1}$.
Since the linear map $A: \PP^{d-1} \to \PP^{n-1}, x \mapsto (\ell_i(x))_{i\in [n]}$  is injective,
these points satisfy $Ax \neq Ax'$.
Hence $Ax$ and $Ax'$ differ only by sign flips. 

Consider any partition $I \sqcup J $ of $[n]= \{1, \ldots, n\}$.
 Let $A_{I,J}$ be the matrix whose $i$th row is $A_i$  if $i \in I$ and $-A_i$ if $i \in J$. 
Then $(Ax)^2 = (Ax')^2$ exactly when $Ax' = A_{I,J}\,x$
for some partition $I \sqcup J = [n]$.
Note that if $|I| \geq d$ or $|J| \geq d$, then $x'$ is forced to equal $x$. 
Therefore the model parametrization is non-injective precisely on the 
subspaces $\ker(BA_{I,J}) \subset \PP^{d-1}$, where the partition
$I \sqcup J=[n]$ satisfies $|I|, |J| \leq d-1$.
Here  $B$ denotes an $(n-d) \times d$ matrix 
with ${\rm ker}(B) = {\rm im}(A)$.
The subspaces $\ker(BA_{I,J})$  have dimension $2d - n - 1$ as desired. 

In the boundary case $n=2d-2$, we have $|I| = |J| = d-1$.
There are precisely $\frac{1}{2}\binom{2d-2}{d-1}$ such partitions  $I \sqcup J = [n]$.
Each of them contributes a line to the singular locus of $X_{d,n}$.
\end{proof}

\begin{example}[$d=4$]
The variety $X_{4,n}$ is smooth for $n \geq 7$. For $n=6$, our model is a
threefold of degree $8$ in $\PP^5$, defined by one quartic, ten quintics and two sextics.
The singular locus of  $X_{4,6}$ consists of ten lines, one for each
partition of $\{1,2,\dots,6\}$ into two triples.

 For $n=5$,
we take $\ell_i = x_i$ for $i=1,2,3,4$ and $\ell_5 = x_1+x_2+x_3+x_4$.
Then $I_{4,5}$ is generated by one polynomial of degree $8$ with
$495$ terms. The singular locus of the threefold $X_{4,5} $ 
decomposes into ten irreducible surfaces in $\PP^4$. Each of these is a quadratic cone, like
$$ V(p_1-p_2, \,p_3^2+p_4^2+p_5^2-2 p_3 p_4-2 p_3 p_5-2 p_4 p_5) \quad = \quad
\hbox{image of} \,\,\, \bigl[V(\ell_1,\ell_2) + V(\ell_3,\ell_4,\ell_5) \bigr]. $$
It would be interesting to derive a combinatorial rule for the singular locus of $X$ in general.
\end{example}

\section{Likelihood Correspondence}
\label{sec4}

In this section, we study the likelihood correspondence of the squared linear model
given by $n$ linear forms $\ell_1,\ldots,\ell_n$ in $d$ variables $x=(x_1,\ldots,x_d)$.
The likelihood correspondence is a central object in algebraic statistics \cite{HS, KSSW}.
It captures the relationship between data and critical points of the log-likelihood function
\begin{equation}
\label{eq:loglike} \Lambda_\mathcal{A}(s,x) \,\, = \,\,
\sum_{i=1}^n s_i \,{\rm log}(p_i(x))  \,\,=\,\,
\sum_{i=1}^n s_i \,{\rm log}(\ell_i^2(x))
\,  - \, \,{\rm log}(q(x)) \sum_{j=1}^n s_j .
\end{equation}
As before, $q = \ell_1^2 + \cdots + \ell_n^2$ is the partition function, and
the parameters $s_1,\ldots,s_n $ represent the data.
The partial derivatives of $\Lambda_\mathcal{A}$ with respect to
$x_1,\ldots,x_d$ are homogeneous rational functions. These
depend linearly on $s_1,\ldots,s_n$. Our object of interest is the
variety in the product space of all pairs $(s,x)$ which is defined by setting these rational functions to zero:
\begin{definition}\label{def:likelihoodCorr}
  The {\em likelihood correspondence} $\mathcal{L}_{\mathcal{A}}$ is the Zariski closure in $\PP^{n-1}\times \PP^{d-1}$ of
  \[
    \left\lbrace (s,x)\in \PP^{n-1}\times \PP^{d-1}  \,\colon\,
       \ell_1(x) \cdots \ell_n(x) q(x) \not= 0,\,\, p(x)\in X_{\mathrm{reg}} \,\,{\rm and} \,\,\,\,
      \frac{\partial \Lambda_\mathcal{A}}{\partial x_i}(s,x)\,=\, 0 \,\, \forall i\in [d] 
             \right\rbrace.
  \]
  The \emph{likelihood ideal} $I_\mathcal{A} \subset \CC[s_1,\dots,s_n,x_1,\dots,x_d]$ is the bihomogeneous prime ideal of~$\mathcal{L}_\mathcal{A}$.
\end{definition}

The projection $\,\PP^{n-1} \times \PP^{d-1} \rightarrow \PP^{n-1}, \,(s,x) \mapsto s \,$
induces a finite-to-one map $\,\mathcal{L}_\mathcal{A} \rightarrow \PP^{n-1}$.
The degree of this map is the ML degree. By Corollary \ref{cor:MLdegree}, the ML degree is
equal to the number of regions of the hyperplane arrangement. Thus,
all fibers are fully real.  
Each connected component of $\PP^{d-1}_\R \backslash \mathcal{A}$
is represented by one point in each fiber of $\,\mathcal{L}_\mathcal{A} \rightarrow \PP^{n-1}$. 

Let $A$ denote the $n \times d$ Jacobian matrix of our linear forms. Thus,  the rows of $A$ 
are the gradient vectors $A_i = \nabla_x \ell_i(x)$ for $i=1,\ldots,n$.
In the following theorem, we assume that $A$~is generic.
In particular, the rank $d$ matroid of $A$ is uniform,
and $q$ is transverse to the generic arrangement $\mathcal{A}$ in $\PP^{d-1}$.
Fix any $(n-d) \times n$ matrix $B$ whose kernel equals the image of $A$.
The $n-d$ rows of $B$ span the linear relations among the
linear forms $\ell_1(x),\ell_2(x),\ldots,\ell_n(x)$.

Our main result in this section is the following explicit description of the prime ideal $I_\mathcal{A}$.

\begin{theorem} \label{thm:likely}
The likelihood ideal for the generic model $X_{d,n}$ is minimally generated by the maximal minors of the matrix
\begin{equation}
\label{eq:likelihoodmatrix} M \,=\, \begin{pmatrix}
s_1 \quad s_2  \quad \cdots \quad s_n \smallskip \\
\ell_1^2  \quad \ell_2^2 \quad \cdots \quad \ell_n^2 \smallskip \\
 B \cdot {\rm diag}(\ell_1,\ldots,\ell_n)
\end{pmatrix}.
\end{equation}
This matrix has $n-d+2$ rows and $n$ columns, and its $\binom{n}{d-2}$ minors all have degree $(1,n-d+2)$.
\end{theorem}

\begin{proof}
We use the description of the likelihood ideal for parametric models appearing  in \cite[Definition 2.3]{AandL}
and \cite[Section 1]{KSSW}. Therein, the authors consider
polynomials $f_1,\ldots,f_m$.  When these define a strict normal crossing (SNC) arrangement,
the construction in \cite[equation~(7)]{KSSW} gives a matrix $Q=Q^s_{\backslash 1}$ 
whose maximal minors generate the likelihood ideal.
In that construction, the first polynomial $f_1$ is special.
Namely, the subscript ``$\backslash 1$'' indicates that the column for $f_1$ was deleted from 
another matrix $Q^s$.
In what follows, the partition function $q$ plays the role of $f_1$, and
the linear forms $\ell_1,\ldots,\ell_n$ play the role of $f_2,\ldots,f_m$.

  The log-likelihood function of the squared linear model $X_{d,n}$ is shown in (\ref{eq:loglike}).  
   This function coincides with the log-likelihood function of the (unsquared) arrangement $\mathcal{A}' = \bigl\{ q,\ell_1,\ell_2,\dots,\ell_n \bigr\}$ after substituting 
   $s_0 \mapsto -(s_1 + \cdots +s_n)$ and   $s_i \mapsto 2s_i$ for $i=1,\dots,n$.
   We claim that the arrangement $\mathcal{A}'$ is SNC in the sense of \cite{KSSW}.
   This means that all intersections of hypersurfaces in $\mathcal{A}'$ are smooth of the expected dimension.
   
   Indeed,  by Lemma \ref{lem:generic_quadric}, no intersection of the hyperplanes is tangent to $V(q)$. It remains to show that the hyperplanes meet $V(q)$ in the expected dimension. Without loss of generality, we may work in an affine chart and transform coordinates so that $\ell_i(x) = x_i$ for $i=1,\dots,d$. Assume there exists $x \in V(q) \cap \bigcap_{i=1}^{d-1} V(\ell_i)$. Then $x= (0,\dots,0,x_d)$ and $q(x) = \alpha x_d^2$ for some nonzero constant $\alpha$. Hence, $x\in V(\ell_d)$, so $x$ is in the intersection of $d$ hyperplanes. This is a contradiction to the $\ell_i$ being in general position. Therefore, $\mathcal{A}'$ is strict normal crossing.

It now follows from \cite[Corollary 2.4]{KSSW} that the likelihood ideal $I_\mathcal{A}$ is
generated by the maximal minors of the following matrix, which has $n+2$ rows and $n+d$ columns:
\begin{equation}
\label{eq:likelihoodmatrix2}
    Q \,\,= \,\,\begin{pmatrix}
      0 & 0 & \dots & 0 &  \partial_{x_1} q(x) & \partial_{x_2} q(x) & \dots & \partial_{x_d} q(x) \\
      \ell_1(x) & 0 & \dots & 0 &  \partial_{x_1} \ell_1(x) & \partial_{x_2} \ell_1(x) & \dots & \partial_{x_d} \ell_1(x) \\
      0 & \ell_2(x) & \dots & 0 & \partial_{x_1} \ell_2(x) & \partial_{x_2} \ell_2(x) & \dots & \partial_{x_d} \ell_2(x) \\
      \vdots & \vdots & \ddots & \vdots & \vdots &   \vdots & \ddots & \vdots \\
      0 & 0 & \dots & \ell_n(x) &  \partial_{x_1} \ell_n(x) & \partial_{x_2} \ell_n(x) & \dots & \partial_{x_d} \ell_n(x) \\
     s_1 &s_2 & \dots & s_n &  0 & 0 & \dots & 0 \\
    \end{pmatrix}.
\end{equation}
On the right, we see the $n \times d$ Jacobian matrix
$\,A = \begin{pmatrix} \partial_{x_i} \ell_j(x) \end{pmatrix} \,$
of the $n$ linear forms. And, the row above this is the gradient
$\,\nabla q(x) =  2x^T A^TA$  of the quadratic form $q(x)$.

The image of $A$ spans the kernel of
the $(n-d) \times n$ matrix $B$, so
the $n \times n$ matrix
\begin{small} $\begin{pmatrix} B^T \, A \end{pmatrix}^T \,$\end{small} is invertible.
The following $(n+2) \times (n+2)$ matrix has an inverse with 
polynomial entries:
 $$ 
 T \,\, = \,\,
 \begin{pmatrix}
\,\
\,\, 0 \,&\, 0 &  \!\! \cdots\!\! & 0  \, & \, 1\,\, \\
-1 & \,\ell_1(x) & \!\! \cdots\!\! & \ell_n(x) \,&\, 0 \,\,\\
\,\, 0 &  & B & & \, 0 \,\,  \\
\,\, 0 & & A^T & & \, 0 \,\, 
\end{pmatrix}.
$$
Hence the ideal of maximal minors of $Q$ agrees with that of the
$(n+2) \times (n+d)$  matrix
$$ T \cdot Q \,\, = \,\,
\begin{pmatrix}
s_1  \,\,\,\,s_2 \,\,\,\, \cdots \,\,\,\, s_n \,\,\,\,\,&\, 0 \,\,\, 0\, \,\, \cdots\,\,\,  0 \,\\
\ell_1^2 \,\,\,\, \ell_2^2 \, \,\,\,\cdots\, \,\,\, \ell_n^2 \,\,\,\,\,& -x^TA^TA \\
 B \cdot {\rm diag}(\ell_1,\ldots,\ell_n) & \,0 \,\,\,0 \,\,\,\cdots\,\,\, 0 \, \smallskip \\ 
 A^T \cdot {\rm diag}(\ell_1,\ldots,\ell_n) & A^T \cdot A \\
\end{pmatrix}.
$$

Since the rows of $A^T$ are linearly independent,
the  symmetric $d \times d $ matrix $A^T A $ is~invertible.
Hence we can replace $T\cdot Q$ by its submatrix 
given by the first $n$ columns and the first $2+n-d$ rows.
That submatrix is precisely the $(n-d+2) \times n$ matrix 
in (\ref{eq:likelihoodmatrix}). We have thus shown that
$I_\mathcal{A}$ is generated by the  maximal minors of (\ref{eq:likelihoodmatrix}).
All generators have the same degree $(1,n-d+2)$. And, unlike $Q$, the matrix (\ref{eq:likelihoodmatrix})
has no constant entries. This ensures that the 
$\binom{n}{d-2}$ maximal minors of (\ref{eq:likelihoodmatrix})
are  minimal generators for the prime ideal $I_\mathcal{A}$.
\end{proof}

Our argument confirms that the ML degree of 
$X_{d,n}$ equals the ML degree of the arrangement  $\mathcal{A}'$. According to \cite{KSSW},
the latter is the coefficient of $z^{d-1}$ in the generating~function
\begin{equation}
\label{eq:genfcn}
    \frac{1}{(1-z)^{n-d}\,(1-2z)}.
\end{equation}
This coefficient is found to be $\,\mu = \sum_{i=0}^{d-1} \binom{n-1}{i}$, which is the number of regions in
$\PP^{d-1}_\R \backslash \mathcal{A}$.

\begin{example}[$d=3,n=4$] \label{ex:34revisited} Let $\mathcal{A}$ be the
arrangement of four lines in Example \ref{ex:d=3n=4}.~Then
$$ A \,\,=\,\, \begin{pmatrix} 1 & 0 & 0 & 1  \\
      	      	       0 & 1 & 0 & 1 \\
                       0 & 0 & 1 & 1 \end{pmatrix}^{\!\! T}
\,\quad {\rm and}  \quad                       \,
B \,\,=\,\, \begin{pmatrix} 1 & 1 & 1 & -1  \end{pmatrix}.
$$
The likelihood ideal $I_\mathcal{A}$ in $\CC[s_1,s_2,s_3,s_4,x_1,x_2,x_3]$
is generated by the $3 \times 3$ minors of
$$
M \,\,=\,\,
\begin{pmatrix}
s_1 & s_2 & s_3 & s_4 \\
\ell_1^2 & \ell_2^2	& \ell_3^2 & \ell_4^2 \\
\ell_1 & \ell_2 & \ell_3 &\!\! -\ell_4
\end{pmatrix}.
$$
The likelihood correspondence $\mathcal{L}_\mathcal{A}$ is a threefold in $\PP^3 \times \PP^2$.
The map  onto $\PP^3$ is $7$-to-$1$.
\end{example}

We conclude with a remark on the non-generic case.
If $\mathcal{A}'$ is not SNC then the maximal minors of the
matrices (\ref{eq:likelihoodmatrix}) and (\ref{eq:likelihoodmatrix2}) 
generate the same ideal  $I$ in $\CC[s,x]$. However, that ideal is
strictly contained in the likelihood ideal $I_\mathcal{A}$.
We can compute $I_\mathcal{A}$ from $I$ by saturation
with respect to $ \ell_1 \ell_2 \cdots \ell_n$. This follows from
Proposition 2.9 and Remark 2.10 of~\cite{AandL}.

\begin{example}[The braid arrangement] \label{ex:braid2}
Let  $d=3,n=6, c=4$ and consider the model in Example~\ref{ex:braid}.
Setting $ x_1 = x ,\,x_2=y,\,x_3=z,\,x_4 = 0$ and relabeling,
the matrix (\ref{eq:likelihoodmatrix}) equals
$$ M \,\,= \,\,
\begin{pmatrix}
     \,\,s_{12} & \,s_{13}  & s_{14} & s_{23} & s_{24} & s_{34} \\
\,\,     x^2 & \,y^2 & z^2 & (x-y)^2 & (x-z)^2 & (y-z)^2 \\
     -x &\, y & 0 &\! x-y &  0  &  0 \\
     -x & \,0 & z &  0  &\! x-z &  0 \\
    \,\,  0 & \!\! -y &  z &  0  &   0 & \!y-z
      \end{pmatrix}.
$$      
The six maximal minors of $M$ generate a radical ideal $I$. It has the prime decomposition
$$ I \,\, = \,\, I_\mathcal{A}\,\cap \,
\langle x,y \rangle \,\cap \,
\langle x,z \rangle \,\cap \,
\langle y,z \rangle \,\cap \,
\langle x-y,x-z \rangle .
$$
The likelihood ideal $I_\mathcal{A}$ has three minimal generators,
of bidegrees $(1,3)$, $(1,4)$ and $(1,5)$.
\end{example}

\section{Likelihood Degenerations}
\label{sec5}

In this section, we examine the likelihood correspondence of squared linear models over
degenerate data points $s$. In particular, we determine
tropical limits of the critical points of~(\ref{eq:loglikelihoodfn}).
The framework of tropical likelihood degenerations was developed 
for linear models in \cite{ABFKST, AEP}
and for toric models in~\cite{BDH}. 
It allows one to study maximum likelihood estimation as model parameters or data points approach special values, particularly model zeros.
We extend these recent advances on the tropical geometry of likelihood inference.

The following formulation of our MLE problem will be used in this section.
Our model is specified by two matrices $A \in \R^{n \times d}$ and $B \in \R^{(n-d) \times n}$
such that the image of $A$ equals the kernel of $B$.
We shall assume that the pair $(A,B)$ is generic, in a sense to be made precise.

For our computations, we use the $n$ unknowns
$ y_i = \sqrt{p_i} = \ell_i(x) $, $\,i=1,2,\ldots,n$.
Thus we work in  projective space $\PP_\R^{n-1}$ with 
coordinates $y = (y_1:y_2:\cdots:y_n)$.
The hyperplane arrangement $\mathcal{A}$ is 
the restriction of the coordinate hyperplanes in $\PP_\R^{n-1}$ to ${\rm im}(A) \simeq \PP_\R^{d-1}$.

According to Theorem \ref{thm:likely}, our task is to  solve the system of polynomial equations
\begin{equation}
\label{eq:oursystem}
\quad  B \cdot y \, = \, 0 \,\quad \,{\rm and} \quad\, {\rm rank} \begin{pmatrix}
s_1 \quad s_2  \quad \cdots \quad s_n \smallskip \\
y_1^2  \quad y_2^2 \quad \cdots \quad y_n^2 \smallskip \\
 B \cdot {\rm diag}(y_1,\ldots,y_n)
\end{pmatrix} \,\leq\, n-d+1. 
\end{equation}
 For generic data $s = (s_1,\ldots,s_n) \in \PP^{n-1}_\R$, the number of solutions to 
 (\ref{eq:oursystem}) is $\mu = \sum_{i=0}^{d-1} \binom{n-1}{i}$.
 
We set $[n] = \{1,2,\ldots,n\}$ and we write 
$e_i$ for the $i$th standard basis vector in $\R^n$.
In what follows next, we specialize $s=e_i$. This means that $s_i = 1$ and $s_j=0$ for $j \in [n] \backslash \{i\}$.

\begin{theorem} \label{thm:happydegeneration}
Suppose that $(A,B)$ is generic and set $s=e_i$.
Then (\ref{eq:oursystem}) defines a radical ideal in $\R[y_1,\ldots,y_n]$,
which has $\mu$ zeros $y$ in $\PP^{n-1}_\R$.
The supports of the $\mu$ zeros are distinct.
Namely, $J = \{j \in [n]: y_j = 0\}$ ranges
over all subsets of
cardinality at most $d-1$ in $[n] \backslash \{i\}$.
\end{theorem}

\begin{proof}
  Let $\mathcal{A}^{(i)}$ be the affine arrangement of $n-1$ hyperplanes
  obtained by setting $y_i=1$.
  Let $b_j$ denote the $j$th column of the matrix $B$.
  We consider the $(n-d+1) \times (n-1)$ matrix
  $$ B^{(i)} \,\,= \,\,\begin{pmatrix}
y_1 & \cdots & y_{i-1} & y_{i+1} & \cdots & y_n \\
b_1 & \cdots & b_{i-1} & b_{i+1} & \cdots & b_n 
\end{pmatrix}.
  $$
  We also introduce the  diagonal matrix
   $Y^{(i)} = {\rm diag}(y_1, \ldots,y_{i-1},y_{i+1},  \ldots, y_n)$.
   Since $s=e_i$, the    maximal minors of the matrix in
  (\ref{eq:oursystem}) are precisely the maximal minors of $B^{(i)} Y^{(i)}$.
  
  After relabeling   we may assume $i = n$.
  Each maximal minor of $B^{(n)}Y^{(n)}$
    factors as $\,\det(B^{(n)}_C) \,y_{c_1} y_{c_2} \cdots\, y_{c_{n-d+1}}\,$ where $B^{(n)}_C$ is the submatrix of $B^{(n)}$ 
    with column indices $C = \{c_1, \ldots, c_{n-d+1}\}\subseteq [n-1]$. 
    We claim that the ideal of maximal minors $I_{n-d+1}(B^{(n)} Y^{(n)})$ is radical and that its prime decomposition is given by
\begin{equation}
\label{eq:primedecomposition}
   I_{n-d+1}(B^{(n)} Y^{(n)}) \,\,= \,\,\bigcap_{J} \,\biggl[\, \langle y_j \colon j \in J \rangle 
   \,+\, I_{n-d+1}( B^{(n)}_{[n-1] \setminus J}) \, \biggr].
\end{equation}
  Here, the intersection is taken over all flats $J \subseteq [n-1]$ of the affine arrangement $\mathcal A^{(n)}$. 
  Since $B$ is generic, there are $\mu$ flats,
    given by the subsets $J $ of cardinality $\leq d-1$.
  The factorization of minors above shows that the radical of the ideal on the left in  (\ref{eq:primedecomposition})
  equals the intersection on the right. In particular,
  the ideal $I_{n-d+1}(B^{(n)} Y^{(n)})$
  in (\ref{eq:primedecomposition})
   has the expected codimension $d-1$. Since it is a determinantal ideal, this implies it is Cohen--Macaulay \cite[Theorem 18.18]{Eis} of degree $\mu$. Therefore, both ideals in \eqref{eq:primedecomposition} are Cohen--Macaulay of the same degree and agree up to radical, hence they are equal.

  We now impose the $n-d$ linear equations $B \cdot y = 0$ on each of the 
  $\mu$ linear spaces of codimension $d-1$  on the right in (\ref{eq:primedecomposition}).
  Since $B$ is generic,
  this  has a unique solution $y \in \PP^{n-1}$, and this solution satisfies $y_j \not=0$ for $j \not\in J$.
  More precisely, the coordinates are
  $$ y_n = 1, \,\,\, \hbox{$y_j = 0$ \ for \ $j \in J$}, \quad
  {\rm and} \quad y_{[n-1] \setminus J} \,=\, - B_{[n-1]\setminus J}^T
 \, (B_{[n-1]\setminus J}B_{[n-1]\setminus J}^T)^{-1} \,b_n. $$
  For generic $B$, $y_{[n-1] \setminus J}$ has non-zero coordinates,
  and    $\,B \cdot y = B_{[n-1] \setminus J} \cdot y_{[n-1] \setminus J} + b_n = 0. $

We have shown that the linear space $\{y \in \PP^{n-1}: B \cdot y = 0\}$
intersects the $\mu$ irreducible components in distinct reduced points. Therefore,
the  entries of $B \cdot y $ form a regular sequence modulo $I_{n-d+1}(B^{(n)}Y^{(n)})$.
    Adding this regular sequence yields a Cohen--Macaulay ideal of 
  Krull dimension one in $\R[y_1,\ldots,y_n]$. This ideal is radical, and it equals the ideal
  in (\ref{eq:oursystem}).
  \end{proof}
  
\begin{example}[$d=2,n=4$] \label{ex:nongeneric-uniform}  The rank $2$ matroid on $[4]$ is uniform for the model given by
$$ A^T \, = \,
\begin{pmatrix}
1 & 1 & 1 & 0 \\
0 & 1 & 2 & 1 
\end{pmatrix} \quad {\rm and} \quad
B \,= \,
\begin{pmatrix}
 1 & -2 & 1 & 0 \\
 1 & -1 & 0 & 1 
\end{pmatrix}.
$$
Theorem \ref{thm:happydegeneration} concerns
 the special data  $s = (1,0,0,0)$. The relevant minor of $M$ in (\ref{eq:likelihoodmatrix})~is
$$ \det(B^{(1)} Y^{(1)})\,\, =\,\, \det \begin{pmatrix}
        y_2^2 & y_3^2 & y_4^2 \, \,\\
          -2 y_2 & y_3 &  0  \, \,\\
       -y_2 & 0 &   y_4 \, \,  \end{pmatrix} .$$
The radical ideal $\,I_3( B^{(1)} Y^{(1)}) \,=\, \langle y_2 \rangle \,\cap \, 
  \langle y_3 \rangle \,\cap \, \langle y_4 \rangle \,\cap \,     
  \langle y_2+2y_3+y_4  \rangle\,$ gives $\mu = 4$ planes in~$\PP^3$, but two of these planes meet the line
   $\{y \in \PP^3: B \cdot y = 0\}$ in the same point. So the
   model is not generic in the sense of Theorem \ref{thm:happydegeneration}.
   The likelihood ideal at $s=e_1$ is not~radical:
     $$ \! I_3( B^{(1)}Y^{(1)}) + \langle B \cdot y \rangle =
\langle y_2,\, y_1+y_4,\,y_3-y_4 \rangle \,\cap\,
\langle y_3^2,\, y_1-y_3+2 y_4,\,y_2-y_3+y_4 \rangle \,\cap\,
\langle y_4,\,y_1-y_3,\,y_2-y_3 \rangle.
$$
One checks that this ideal becomes radical for any small perturbation of the model $(A,B)$.
\end{example}

We now discuss tropical MLE.
By Theorem \ref{thm:two}, all  solutions to (\ref{eq:oursystem}) are real, and there is one solution in each region of $\PP^{d-1}_{\R} \backslash \mathcal{A}$.
This statement remains valid for any real closed field, such as the
field $\,R = \R\{\!\{ \epsilon \} \! \}\,$ of real Puiseux series.
If our data vector $s$ has coordinates in $R$, the system (\ref{eq:oursystem}) has $\mu$ solutions with coordinates in $R$. 
The coordinatewise valuation $w = {\rm val}(s) \in \QQ^n$ of $s$ is the \emph{tropical data vector}.
For any solution $y$ of  (\ref{eq:oursystem}) in $\PP^{n-1}_R$,
we also record its valuation
 $z = {\rm val}(y) \in \QQ^n/\QQ {\bf 1}$.
We present a formula that writes $z$ in terms of $w$.

\begin{corollary} \label{cor:tropgeneric}
Fix a generic model $(A,B)$ and $i \in [n]$. Given data $s \in R^n$, where
$w = {\rm val}(s)$ satisfies $\,w_i < w_j\,$ for $\,j \in [n] \backslash \{i\}$,
the $\mu$ tropical critical points $z$ are distinct. They are
\begin{equation} \label{eq:predicted}
 z \,\,=\,\,  \sum_{j \in J} \,(w_j - w_i)\, e_j   , \end{equation}
where $J$ runs over the $\mu$ subsets of cardinality at most $d-1$ in $[n] \backslash \{i\}$.
\end{corollary}
 
 Before we come to the proof, we go over a detailed example that explains the assertion.
 
 \begin{example}[$d=3,n=4$] \label{ex:sevenarcs}
  Let $B = \begin{pmatrix} 1 & 1 & 1 & \! -1 \end{pmatrix}$.
This is the Steiner surface model in Examples \ref{ex:d=3n=4}
and \ref{ex:34revisited}. Its ML degree is $\mu=7$.
We examine  Theorem \ref{thm:happydegeneration} for  $i=1$, so the data vector is $s = (1,0,0,0)$.
The ideal in (\ref{eq:oursystem}) is radical.
Its seven solutions in $\PP^3_\R$ are
$$ \begin{matrix} y(0) &  = &
(1:0:0:1),\,
(1:0:-1:0),\,
(1:-1:0:0) ,
\qquad \qquad \qquad \qquad \qquad \qquad\,  \\ & & 
(2:0:-1:1),\,
(2:-1:0:1), \,
(2:-1:-1:0), \, (3:-1:-1:1).
\end{matrix}
$$

We now introduce a small parameter $\epsilon > 0$, and we fix the
data vector $s = \bigl(1,\epsilon^3, \epsilon^4,\epsilon^5 \bigr)$.
Each solutions  $y(\epsilon) \in \PP^3_R$ 
converges to one of the $y(0)$ above. The seven solutions $y(\epsilon)$ are
$$ \begin{matrix}
\bigl( \,1\,:\,
2\epsilon^3 + 6\epsilon^6{-}2\epsilon^7{+}2\epsilon^8 \,:\,
2\epsilon^4 - 2\epsilon^7{-}6\epsilon^8{+}2\epsilon^9 \,:\,
 1 +2\epsilon^3{+}2\epsilon^4{-}6\epsilon^6{-}4\epsilon^7 \,\bigr),
 \smallskip \\
\bigl(\,1\,:\,
2\epsilon^3 - 6\epsilon^6{+}2\epsilon^7{-}2\epsilon^8\,:\,
-1 - 2\epsilon^3{-}2\epsilon^5{+}6\epsilon^6{-}2\epsilon^7\,:\,
-2\epsilon^5 +2\epsilon^8{-}2\epsilon^9{+}6\epsilon^{10}\,\bigr),
\smallskip \\
\bigl(\,1\,:\,
 -1-2\epsilon^4{-}2\epsilon^5{-}2\epsilon^7{+}4\epsilon^8\,:\,
  2\epsilon^4+2\epsilon^7{-}6\epsilon^8{-}2\epsilon^9 \,:\,
   -2\epsilon^5-2\epsilon^8{+}2\epsilon^9{+}6\epsilon^{10} \,\bigr),
\smallskip \\
\bigl( 2\,:\,
6\epsilon^3-48\epsilon^6{+}12\epsilon^7{+}12\epsilon^8{+}816 \epsilon^9\,:\,
{-}1{-}3\epsilon^3{-}3\epsilon^4{+}3\epsilon^5{+}24\epsilon^6\,:\,
1{+}3\epsilon^3{-}3\epsilon^4{+}3\epsilon^5{-}24\epsilon^6\,\bigr),
\smallskip \\
\bigl(\,2\,:\, -1-3\epsilon^3{-}3\epsilon^4{+}3\epsilon^5{+}12\epsilon^6 \,:\,
6\epsilon^4+12\epsilon^7{-}48\epsilon^8{+}12\epsilon^9 {-}12 \epsilon^{10}\,:\,
1-3\epsilon^3{+}3\epsilon^4{+}3\epsilon^5{+}12\epsilon^6 \,\bigr),
\smallskip \\
\bigl(\,2\,:\,
- 1 - 3\epsilon^3{+}3\epsilon^4{-}3\epsilon^5{+}12\epsilon^6\,:\,
- 1+ 3\epsilon^3{-}3\epsilon^4{-}3\epsilon^5{-}12\epsilon^6\,:\,
- 6\epsilon^5-12\epsilon^8{-}12\epsilon^9{+}48\epsilon^{10}\,\bigr),
\smallskip \\
\bigl(\,3\,:\,
 -1-8\epsilon^3{+}4\epsilon^4{+}4\epsilon^5{+}72\epsilon^6\,:\,
  -1+4\epsilon^3{-}8\epsilon^4{+}4\epsilon^5{-}36\epsilon^6\,:\,
   1-4\epsilon^3{-}4\epsilon^4{+}8\epsilon^5{+}36\epsilon^6\,\bigr).
   \end{matrix}
$$ 
Each coordinate is a convergent power series in $\ZZ[\![ \epsilon ] \! ]$.
We computed these with the command {\tt puiseux} in {\tt Maple}. 
Passing to valuations, the tropical data vector is $w = {\rm val}(s) =  (0,3,4,5)$.

The tropical MLE is given by the valuations of the seven solutions above. We see that
$$ \begin{matrix} z \,=\, {\rm val}(y(\epsilon)) \,= \,
(0:3:4:0), \,(0:3:0:5),\, (0:0:4:5) , \qquad \qquad \qquad \quad \\ \qquad \qquad
\qquad \qquad (0:3:0:0),\, (0:0:4:0),\, (0:0:0:5),\, (0:0:0:0) . \end{matrix} $$
These are the $\mu=7$ tropical critical points. These were written in terms of $w$ in (\ref{eq:predicted}).
 \end{example}
 
 \begin{proof}[Proof of Corollary \ref{cor:tropgeneric}]
We may assume  $s_i = 1$ and $w_i = 0$.
Then $s_1,\ldots,s_{i-1},s_{i+1},\ldots,s_n$ have positive valuations.
We write $w= \sum_{j\in[n]} w_j e_j = \sum_{j \in [n]} {\rm val}(s_j) e_j$ for the tropical data.
The system
(\ref{eq:oursystem}) has $\mu$ distinct solutions $y(\epsilon)$ in $\PP^{n-1}_R$.
Their special fibers $y(0) \in \PP^{n-1}_\R$ are the zeros of
the prime components in (\ref{eq:primedecomposition}).
Hence there are $\mu$ distinct tropical critical points.

Suppose that $y(\epsilon)$ is a tropical critical point with valuation $z = {\rm val}(y(\epsilon))$ and let $J = \{j \in [n] \mid z_j > 0\}$. 
Because $J = \{j \in [n] \mid (y(0))_j = 0\}$,
Theorem~\ref{thm:happydegeneration}  implies that $|J| \leq d - 1$.
By the definition of $J$, $z_j = 0$ if $j \notin J$.
Our claim states that $z_j = w_j$ if $j \in J$. 

Fix $j \in J$.
To prove $z_j = w_j$, we tropicalize a determinantal equation in (\ref{eq:oursystem}).
Let $C$ be the union of $\{i,j\}$ with an $(n - d)$-subset of $[n]\backslash(J \cup \{i\})$.
The minor of $M$ indexed by $C$ is 
$$ \prod_{c \in C}y_c \,\cdot \sum_{k \neq \ell \in C} \! \pm (s_k/y_k) \,y_\ell \,\det(B_{C \backslash \{k, \ell\}}).$$
The coordinates $y_c$ of $y$ are nonzero for $c \in C$, so the sum vanishes. 
After tropicalization, the vanishing of the sum becomes the condition that the minimum $\min_{k \neq \ell \in C} (w_k - z_k + z_\ell)$ is attained twice.
Since $z_k = 0$ for $k \in C \backslash \{j\}$,  we have
\begin{align*}
  w_k - z_k + z_\ell \,\,= \,\,
  \begin{cases}
  \,  w_k       &\quad \text{if $\,k,\ell \neq j$,}\\
  \,  w_j - z_j &\quad \text{if $\,k = j$,}\\
  \,  w_k + z_j &\quad \text{if $\,\ell = j$.}
  \end{cases}
\end{align*}
Since $z_j > 0$, the minimum is never attained in the last case.
By assumption, $w_k > w_i$. This means that
$w_k$ cannot attain the minimum unless $k = i$.
Hence the tropical equation simplifies to $\min(w_i, w_j - z_j)$.
This yields $w_j - z_j = w_i = 0$, and we conclude that $z_j = w_j$.
\end{proof}

The field $R$ of real Puiseux series 
is both valued and ordered, and it is
interesting to combine both structures
when studying the likelihood geometry of squared linear models.
In the setting of Corollary \ref{cor:tropgeneric}, 
there is one critical point $y $ in each region
of $\PP^{n-1}_R \backslash \mathcal{A}$.
The region is a polytope in $\PP^{n-1}_\R$, and
$y = y(\epsilon)$  is an arc whose limit $y(0)$ lies
on one of the faces of that polytope.
All arcs are repelled by the hyperplane  $\{y_i = 0\}$ at infinity,
so they converge to a face of the affine arrangement $\mathcal{A}^{(i)}$.
Note that $\mathcal{A}^{(i)}$ has $\mu$ flats,
while $\mathcal{A}$ has $\mu$ regions.
The $\mu$ critical arcs $y = y(\epsilon)$ therefore
specify a bijection between the flats and the regions.

\begin{figure}[h]
  \centering
  \begin{tikzpicture}[scale=0.9]
    \draw[thick] (-5,0) -- (5,0);
    \node at (5.25,0) {$y_2$};
    \draw[thick] (0,-2) -- (0,4);
    \node at (0, 4.25) {$y_3$};    
    \draw[thick] (4,-2) -- (-2,4);
    \node at (-2, 4.25) {$y_4$};        
    \draw[thick] (-5,-2) -- (-7,4);
    \node at (-7, 4.25) {$y_1$};            
    \node at (0.85,0.25) {$+--+$};
    \node at (2,3) {$+---$};
    \node at (5, -1) {$++--$};
    \node at (1,-1.5) {$++-+$};
    \node at (-2,-1) {$++++$};
    \node at (-2,1.5) {$+-++$};
    \node at (-1,4) {$+-+-$};

    % y2 and y3
    \draw[blue, <-, thick,domain=0.2:0.6,samples=100,variable=\t]
    plot ( {-4*\t^4 + 4*\t^7 + 12*\t^8 - 4*\t^9}, {-4*\t^3 - 12*\t^6 + 4 * \t^7 - 4*\t^8});

    % y2 and y4
    \draw[blue, <-, thick,domain=0.2:0.6,samples=100,variable=\t]
    plot ({2 + 4*\t^3 + 4*\t^5 - 12*\t^6 + 4*\t^7}, {-4*\t^3 + 12*\t^6 - 4*\t^7 + 4*\t^8});

    % y3 and y4
    \draw[blue, thick, <-, domain=0.3:0.6,samples=100,variable=\t]
    plot ({-4*\t^4 - 4*\t^7 + 12*\t^8 + 4*\t^9}, {2 + 4*\t^4 + 4*\t^5 + 4*\t^7 - 8*\t^8});

    % y2
    \draw[blue, thick, <-, thick,domain=0.2:0.4,samples=100,variable=\t]
    plot ({1 + 3*\t^3 + 3*\t^4 - 3*\t^5 - 24*\t^6}, {-6*\t^3 + 48*\t^6 - 12*\t^7 - 12*\t^8 - 816*\t^9});

    % y3
    %% cartoon
    \draw[blue, <-, thick,domain=0.2:0.5,samples=100,variable=\t]
    plot ({-6*\t^3 - 12*\t^7 + 144*\t^8 - 12*\t^9 + 12*\t^(10)}, {1 + 3*\t^3 + 3*\t^4 - 3*\t^5 -12*\t^6});

    % y4
    \draw[blue, <-, thick,domain=0.45:0.6,samples=100,variable=\t]
    plot ({1 - 3*\t^3 + 3*\t^4 + 3*\t^5 + 12*\t^6}, {1 + 3*\t^3 - 3*\t^4 + 3*\t^5 -12*\t^6});

    % center
    \draw[blue, <-, thick,domain=0.4:0.47,samples=100,variable=\t]
    plot ({(-2 + 8*\t^3 - 16*\t^4+8*\t^5-72*\t^6)/(-3)},{(-2-16*\t^3 + 8*\t^4 + 8*\t^5 + 166*\t^6)/(-3)});
  \end{tikzpicture}
  \caption{Tropical MLE for $d=3,n=4$ gives a bijection between the seven regions
  of $\PP^2_\R \backslash \mathcal{A}$
    and the seven faces of the triangle.      Each arc $y(\epsilon)$
    travels  from its region to the triangle. 
      }
  \label{fig:fourlines}
\end{figure}
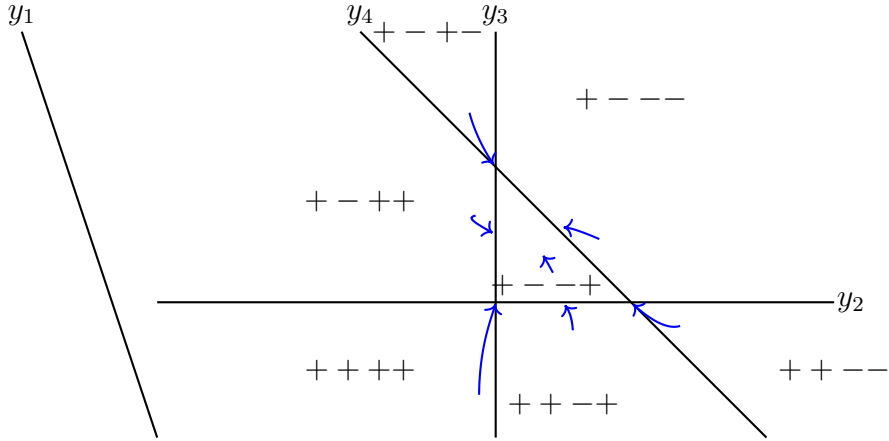

\begin{example}[$n=d+1$]
Here $\mathcal{A}$ has $\mu = 2^d-1$ regions.
Only one of them is bounded in $\mathcal{A}^{(i)}$.
The bounded region is a $(d-1)$-simplex, so it has $\mu$ faces.
Each region meets the $(d-1)$-simplex in a distinct face,
made manifest by an arc $y(\epsilon)$ from the region to that face.

We visualize the case  $d=3$ with $i=1$ in Figure \ref{fig:fourlines}.
The arrangement $\mathcal{A}$ has four lines and seven regions in $\PP^2_\R$.
The seven arcs $y(\epsilon)$ are given algebraically in
Example \ref{ex:sevenarcs}. Each limit point $y(0)$ lies in the
closed triangle: three at the vertices, three on the edges, and one in the~interior.
The seven tropical solutions
    $z$ reveal how each arc approaches its limit point.
\end{example}

\section{Log-Normal Polytopes}
\label{sec6}

The likelihood correspondence can be viewed as
the logarithmic normal bundle of the $(d-1)$-dimensional
variety $X \subset \PP^{n-1}$. This geometric perspective was emphasized  in \cite{HS}.
Its fiber over a regular point $p$ in $X_{>0}$
 is a linear space of
dimension $n-d$ inside the space $\PP^{n-1}$ of data $s$.
The  \emph{log-normal polytope} is the intersection of 
 this fiber with the simplex $\overline{\Delta}_{n-1} = \PP^{n-1}_{\geq 0}$.
Each log-normal polytope has dimension $n-d$.

The model is given by a pair $(A,B)$ as in Section \ref{sec5}.
We fix $p = (y_1^2,\ldots,y_n^2) $ in  $ X_{>0}$. Then
$y = Ax$ for some $x \in \PP^{d-1}_\R \backslash \mathcal{A}$, and we have $By = 0$.
The rank constraints in  (\ref{eq:oursystem}) 
give $d-1$ independent linear equations in
the unknowns $s = (s_1,\ldots,s_n)$. We seek the
solutions to these equations in the simplex $\overline{\Delta}_{n-1} = \PP^{n-1}_{\geq 0}$.
In symbols, the log-normal polytope at $y$ is
\begin{equation}
\label{eq:lnp}
\Pi(y) \,\, := \,\,  \bigl\{ \,s \in \R^n_{\geq 0}  \,: \, s \,\,{\rm satisfies} \,\,(\ref{eq:oursystem})\,\,
{\rm and} \,\, s_1 + s_2 + \cdots + s_n = 1   \bigl\}.
 \end{equation}
Thus, $\Pi(y)$ consists of all
 probability distributions in the row span of the $(n\!-\!d\!+\!1) \times n$~matrix
 \begin{equation}
 \label{eq:BY}
 \begin{pmatrix} 
y_1^2  \,\,\, y_2^2 \,\,\, \cdots \,\,\, y_n^2 \smallskip \\
 B \cdot Y
 \end{pmatrix} \,\,\, = \,\,\, \begin{pmatrix} y_1 \,\,\, y_2 \,\,\, \cdots \,\,\, y_n \smallskip \\ B \end{pmatrix} \cdot Y \,\,\, = \,\,\,
  \begin{pmatrix}
  1 \,\,\, 1 \,\,\, \cdots \,\,\, 1 \smallskip \\
  B \cdot Y^{-1}
       \end{pmatrix}
  \cdot Y^2.
\end{equation}
Here we abbreviate $Y = {\rm diag}(y_1,\ldots,y_n)$.
The $\binom{n}{d-1}$ maximal minors of 
(\ref{eq:BY}) factor as follows:
\begin{equation}
\label{eq:newminors} \quad
y_{j_0} y_{j_1} \cdots y_{j_{n-d}} \cdot
{\rm det}
\begin{pmatrix}
y_{j_0}  \!&\! y_{j_1} \!&\! \!\cdots \!&\! y_{j_{n-d}} \\
b_{j_0} \!&\! b_{j_1} \!&\!\! \cdots \!&\! b_{j_{n-d}} \end{pmatrix} 
\qquad {\rm for} \quad 1 \leq y_{j_0} \!<\! y_{j_1} \!<\! \cdots \!<\! y_{j_{n-d}} \leq n.
\end{equation}
These determinants define $\binom{n}{d-1}$  hyperplanes in
$\PP^{d-1} = \PP({\rm ker}(B))$, in addition to the $n$ given hyperplanes $y_i = \ell_i(x)$.
The {\em chamber arrangement} $\mathcal{A}^{\rm ch}$ consists of all
$n + \binom{n}{d-1}$ hyperplanes.

Using this enlarged arrangement, we obtain the following characterization of our polytope.

\begin{theorem}  \label{thm:chamber}
The polytope $\Pi(y)$ is combinatorially dual to the
convex hull $\,Q$ of the columns of the $(n-d) \times n$ matrix $B Y^{-1}$.
Each log-normal polytope $\Pi(y)$ 
is simple for $y \in \PP^{d-1}_\R \backslash \mathcal{A}^{\rm ch}$.
Its combinatorial type is fixed for $y$ in a 
region of the chamber arrangement $\mathcal{A}^{\rm ch}$.
\end{theorem}

\begin{proof}
Let $Q^\Delta$ denote the polar dual of $Q$.
We prove that $Q^\Delta$ is combinatorially equivalent to $\Pi(y)$.
Since $y_i^2 > 0$ for all $i$, removing the factor $Y^2$ from (\ref{eq:BY}) does not change the combinatorial type of the polytope
given by its columns.
Setting $\tilde{B} = \begin{pmatrix} \bf 1\\ BY^{-1} \end{pmatrix}$, we~have
\begin{align*}
    Q^\Delta 
    &\,=\, \{z \in \R^{n-d} \colon z^TBY^{-1} \leq {\bf 1} \} \,
     \,\simeq \,\{z \in \R^{n-d+1} \colon z_1 = 1,\,\,z^T \tilde{B} \geq 0 \} ,\qquad {\rm and}
    \\
    \Pi(y) &\,=\, \{z^T \tilde{B} \in \R^{n} \colon z^T\tilde{B} \geq 0,  \sum_{i=1}^n (z^T\tilde{B})_i = 1 \}
     \simeq
    \{z \in \R^{n - d+1} \colon z^T\tilde{B} \geq 0, \sum_{i=1}^n (z^T\tilde{B})_i = 1 \}.
\end{align*}
This shows that the cones over $Q^\Delta$ and $\Pi(y)$ are equal. 
We then apply the argument in the proof of \cite[Theorem 4]{Ale} to conclude that the polytopes are combinatorially equivalent. 

The combinatorial type of $Q$ is determined by the rank $n - d + 1$ oriented matroid of $\tilde B$; see \cite[Section 9.1]{BLSWZ}. 
When $y$ is in $\PP_\R^{d-1} \backslash\mathcal{A}^{\rm ch}$, the minors of $\tilde B$ are nonzero and hence the underlying matroid of the oriented matroid is uniform; hence $Q$ is simplicial.
Thus the polar dual $Q^\Delta$ and $\Pi(y)$ are both simple.  
For $y$ in a fixed region of $\PP_\R^{d-1} \backslash\mathcal{A}^{\rm ch}$, the oriented matroid is fixed, 
since changing the matroid  requires crossing one of the hyperplanes.
\end{proof}

\begin{example}[$d=2,n=6$] We consider the $1$-dimensional model $X$ in $\PP^5$ defined by
$$ A \,\, = \,\, \begin{pmatrix} 1 & 1 & 1 & 1 & 1  & 1 \\ 1 & 2 & 3 & 4 & 5 & 6 \end{pmatrix}^{\! T}
\quad {\rm and} \quad
B \,\, = \,\, \begin{small}
\begin{pmatrix}
\,1 & -2 & \phantom{-}1 & \phantom{-}0 & \phantom{-}0 & \phantom{-}0 \,\,\\
\,0 & \phantom{-}1 & -2 & \phantom{-}1 & \phantom{-}0 & \phantom{-}0 \,\,\\
\,0 & \phantom{-}0 & \phantom{-}1 & -2 & \phantom{-}1  & \phantom{-}0 \,\,\\
\,0 & \phantom{-}0 & \phantom{-}0 & \phantom{-}1 & -2 & \phantom{-}1 \,\,
\end{pmatrix}. \end{small}
$$
The chamber arrangement $\mathcal{A}^{\rm ch}$ consists of $12$ points on the circle $\PP^1_\R$.
In addition to the six points given by $y_1,\ldots,y_6$, there are six new points
from the determinants in (\ref{eq:newminors}). For instance, for
$\{j_0,\ldots,j_4\} = \{1,2,3,5,6\}$, we obtain
$\,3y_1+2y_2+y_3-y_5-2y_6\, = \, 3x_1-7x_2$,
which reveals the transition point $(70:30)$ in $\mathcal{A}^{\rm ch} \subset \PP^1_\R$.
For $x \in \PP^1_\R \backslash \mathcal{A}^{\rm ch}$,
the log-normal polytope is a product of simplices, namely
 $\Delta_2 \times \Delta_2$ or $\Delta_1 \times \Delta_3$ or $\Delta_4$,
depending on the oriented matroid of (\ref{eq:BY}).
For an explicit transition, try the points
$ x = (69:30), (70:30), (71:30)$.
\end{example}

We now turn to the {\em log-Voronoi cell} of a point $p = (y_1^2,\ldots,y_n^2)$ in the squared linear model $X_{>0}$.
This is the subset of all data points $s$ in $\Pi(y)$ such that $p$ is the MLE for~$s$.
Log-Voronoi cells for discrete models were introduced by Alexandr and Heaton in~\cite{AH}. 
Theorems 8, 9 and 10 in \cite{AH} identify various models for which each
log-Voronoi cell coincides with its ambient log-normal polytope.
This holds for all linear models. Their log-Voronoi cells were studied in detail by Alexandr in \cite{Ale}.
However, in general, the log-Voronoi cells are non-linear convex bodies
that are strictly contained in their log-normal polytopes.
For an illustration see \cite[Figure 2]{AH}. In what follows, we initiate the study of
 log-Voronoi cells for squared linear models.
We shall see that their geometry is more complicated than that for linear models.

We begin by discussing the simplest model, which is a circle inscribed in a triangle.

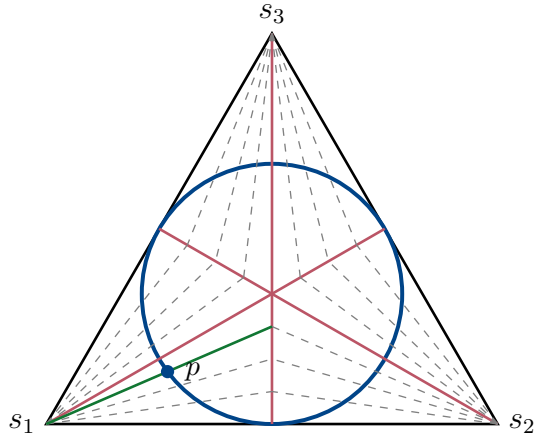
\begin{figure}[h]
  \centering
  \begin{tikzpicture}[scale=1.5]

    \draw [line width=1pt] (0,0) -- (60:4) -- (4,0) -- cycle;

    \coordinate[label=left:$s_1$]  (a) at (0,0);
    \coordinate[label=right:$s_2$] (b) at (4,0);
    \coordinate[label=above:$s_3$] (c) at (2,3.464);
    \coordinate (center) at (2,1.154);

    \draw [line width=1.5pt, color=cb-blue] (center) circle (1.154);
    \draw [line width=1pt, color=cb-red] (a) -- (30:3.464);
    \draw [line width=1pt, color=cb-red] (c) -- (2,0);
    \begin{scope}[shift={(4cm,0cm)}]
      \draw [line width=1pt, color=cb-red] (0,0) -- (150:3.464);
    \end{scope}

    \begin{scope}[xshift=4cm, rotate=120]
        \draw [line width=0.5pt, color=gray, dashed] (b) -- (2,0.865);
        \draw [line width=0.5pt, color=gray, dashed] (b) -- (2,0.577);
        \draw [line width=0.5pt, color=gray, dashed] (b) -- (2,0.2885);
        \begin{scope}[xscale=-1, xshift=-4cm]
          \draw [line width=0.5pt, color=gray, dashed] (c) -- (2,0.865);
          \draw [line width=0.5pt, color=gray, dashed] (c) -- (2,0.577);
          \draw [line width=0.5pt, color=gray, dashed] (c) -- (2,0.2885);
        \end{scope}
    \end{scope}

    \begin{scope}[mirror scope={center={0,0},angle=30}]
      \draw [line width=0.5pt, color=gray, dashed] (a) -- (2,0.865);
      \draw [line width=0.5pt, color=gray, dashed] (a) -- (2,0.577);
      \draw [line width=0.5pt, color=gray, dashed] (a) -- (2,0.2885);
      \mirror;
    \end{scope}

    \begin{scope}[mirror scope={center={0,0},angle=30}]
      \begin{scope}[xscale=-1, xshift=-4cm]
        \draw [line width=0.5pt, color=gray, dashed] (b) -- (2,0.865);
        \draw [line width=0.5pt, color=gray, dashed] (b) -- (2,0.577);
        \draw [line width=0.5pt, color=gray, dashed] (b) -- (2,0.2885);
      \end{scope}
      \mirror;
    \end{scope}

    \draw [line width=1pt, color=cb-green] (a) -- (2,0.865);
    \node [label=right:{\small $p$}, circle, fill=cb-blue, inner sep=0pt, minimum width=5pt] (p) at (1.07448, 0.46471) {};
  \end{tikzpicture}
 \caption{The squared linear model (\textcolor{cb-blue}{blue}) shown inside the triangle $\Delta_2$ of data (black), together with its logarithmic normal bundle (\textcolor{gray}{gray} dashed lines). The triangle is divided into six Weyl chambers (\textcolor{cb-red}{red} lines). The log-Voronoi cell of the point $p$ is the intersection of the fiber of the logarithmic normal bundle with the corresponding Weyl chamber (\textcolor{cb-green}{green} line).}
  \label{fig:logNormalBundleCircle}
\end{figure}

\begin{example}[$n=3,d=2$] \label{ex:circle} Let $\mathcal{A}$ be the arrangement of three
ponts in $\PP^1$ defined by $\ell_1 = x_1, \,\ell_2 = x_2 ,\, \ell_3 = x_1+x_2$.
The model is the conic $X = V(p_1^2+p_2^2+p_3^2-2p_1p_2 -2p_1p_3 -2p_2p_3)$ in $\PP^2$.
In the probability triangle $\overline{\Delta}_2$, this is an inscribed circle, so
it has three connected components in $\Delta_2$. We see this in Figure \ref{fig:logNormalBundleCircle}, 
where the model is drawn in~blue. 

Fix
$p = \bigl(x_1^2: x_2^2: (x_1 + x_2)^2 \bigr) \in X_{>0}$.
The fiber of the log-normal bundle over $p$ is the line 
$$ us_1 + vs_2 + ws_3 \,\,:=\,\,{\rm det} \begin{pmatrix}
\,\,s_1 & \,\,s_2 & s_3 \\ \,\,x_1^2 & \,\,x_2^2 & (x_1\!+\!x_2)^2 \\ -x_1 & \!\!- x_2 & \!x_1\!+\!x_2 \end{pmatrix} .
$$
This is a linear form in $(s_1,s_2,s_3)$, whose coefficients depend cubically on the model point
$$ u \, = \, x_2 (x_1 + 2 x_2) (x_1 + x_2) ,\,\,
v \,=\, - x_1 (x_1 + x_2) (2 x_1 + x_2),\,\,
w \,=\,  - x_1 x_2 (x_1 - x_2) . $$

The triangle $\Delta_2$ is divided into six Weyl chambers, depending on the ordering of $s_1,s_2,s_3$.
These are the six red triangles in Figure \ref{fig:logNormalBundleCircle}.
The log-Voronoi cell at $p $ is the green line segment through $p$.
  It is the intersection of the log-normal line 
 with the Weyl chamber. The log-Voronoi segments interpolate
 between a red boundary and a half-edge of the triangle.
 
This example also shows that each
log-Voronoi cell is strictly contained in its log-normal polytope.
The latter is the intersection of the triangle with the
line spanned by the dashed segment. Thus,
squared linear models are more complicated
than linear models and toric models, for which
these $(n-d)$-dimensional convex bodies coincide
\cite[Theorems 9 and 10]{AH}.
 \end{example}

The topological boundary $\partial S$ of a log-Voronoi cell $S$ consists of data points with at least one additional MLE with another sign pattern. The boundary $\partial S$ is defined by a piecewise analytic function. 
We next identify a setting in which 
the boundary has linear pieces. 
Given $i,j \in \{1,\ldots,n\}$ and $\sigma \in \{-1, +1\}^n$, let $\tau_{ij\sigma}:\PP^{n-1} \to \PP^{n-1}$
denote the map exchanging coordinates $i$ and $j$ followed by coordinatewise multiplication by $\sigma$.   

\begin{proposition} \label{prop:linear_log_Voronoi}
    Fix a squared linear model $X$ and $y^2 \in X$. Let $S$ denote the log-Voronoi cell of $y$.
    If $y$ and $\tau_{ij\sigma}(y)$ have different sign vectors, $\tau_{ij\sigma}(y)^2 \in X$,
    and $V(s_i - s_j) \cap S \neq \emptyset$,   
    then $V(s_i - s_j) \cap S$ is a connected, linear piece of the boundary $\partial S$ of the log-Voronoi cell.
\end{proposition}

\begin{proof}
    Let $S' \subset \Pi(y)$ denote the set of data points $s$ in the log-normal polytope $\Pi(y)$ such that $s$ has an MLE with sign vector equal to the sign vector of $\tau_{ij\sigma}(y)$.
    If $s \in S \cap V(s_i - s_j)$, 
    then the likelihood function for $s$ has the same value at $y$ and at $\tau_{ij\sigma}(y)$. 
    Therefore $s \in S \cap S'$. 
    Since $S$ is convex, $S \cap V(s_i - s_j) \subseteq S \cap S' \subseteq \partial S$ is a connected, linear piece of the boundary.
\end{proof}

If the proposition holds for sufficiently many triples $i,j,\sigma$, then $\partial S$ is piecewise linear.

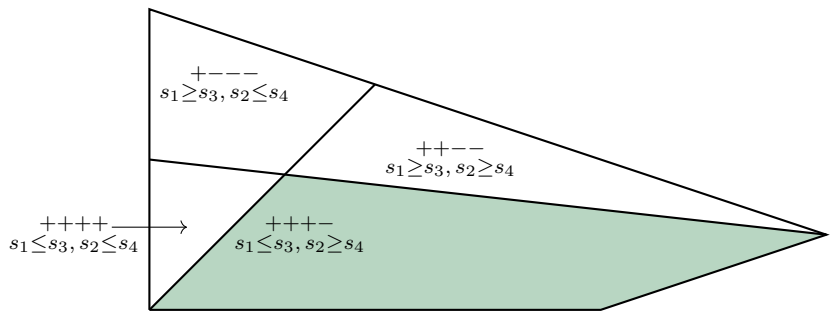
\begin{figure}[t]
    \centering
    \begin{tikzpicture}
        \fill[cb-green!30] (0,0) -- (6,0) -- (9, 1) -- (9/5, 9/5) -- (0,0);
        \draw[thick] (0,0) -- (0,4) -- (9,1) -- (6,0) -- (0,0);
        \draw[thick] (0,0) -- (3,3);
        \draw[thick] (9,1) -- (0,2);

        \node at (2,1) {$\substack{+++-\\s_1 \leq s_3,\, s_2 \geq s_4}$};
        \node at (1,3) {$\substack{+---\\s_1 \geq s_3,\, s_2 \leq s_4}$};
        \node at (4,2) {$\substack{++--\\s_1 \geq s_3,\, s_2 \geq s_4}$};
        \node (a) at (-1,1) {$\substack{++++\\s_1 \leq s_3,\, s_2 \leq s_4}$};
        \draw[->] (-0.5,1.1) -- (0.5, 1.1);
    \end{tikzpicture}
    \caption{The log-normal polygon in Example~\ref{ex:nongeneric-uniform2}. The log-Voronoi cell is marked in \textcolor{cb-green}{green}.}
    \label{fig:linear-log-normal-tiling}
\end{figure}

\begin{example}[$d=2,n=4$] \label{ex:nongeneric-uniform2}  The model in Example~\ref{ex:nongeneric-uniform} has
$ \mathcal{A} = \{x_1,x_1+x_2,x_1+2x_2,x_2\}$. Fix $y = (3, 2, 1, -1)$.
The log-normal polygon $\Pi(y)$ is the intersection of the tetrahedron $\Delta_{3}$ 
with the plane $V(s_1 - s_2 - 3s_3 - 2s_4)$.
This is a quadrilateral, divided into four cells based on the sign vector of the MLE; see Figure~\ref{fig:linear-log-normal-tiling}.
 Their boundaries are $s_1 = s_3$ and $s_2 = s_4$, because the vectors $ \tau_{13,+++-}(y) = (1,2,3,1)$ and $\tau_{24,+---}(y) = (3,1,-1,-2)$ are in the kernel of $B$. 

Consider the slightly modified arrangement $\{x_1,x_1+x_2,x_1-2x_2,x_2\}$.
The log-normal polytope of $y = (3, 2, 5, -1)$ is again a quadrilateral with four nonempty cells. 
The log-Voronoi cell meets all other cells in a codimension 1 boundary. This boundary is nonlinear.
\end{example}

We found that
the log-Voronoi cells of squared linear models exhibit a wide range of behaviors. Experiments suggest that, for a generic model and point, 
the boundary of the log-Voronoi cell is nonlinear. As in Section \ref{sec5}, being generic is stronger than having a uniform matroid.
Example~\ref{ex:nongeneric-uniform2} underscores this.
As in \cite[page 9]{AH}, we think that these boundaries are generally not algebraic.
Experiments indicate that it is also possible to have linear and nonlinear boundary pieces at the same time. Finally, it is possible that all boundaries are linear, so the log-Voronoi cell is a polytope
(see Figure \ref{fig:linear-log-normal-tiling}).
In conclusion, the detailed study of log-Voronoi cells for squared linear models is a promising
direction for future research.

\section{Determinantal Point Processes}
\label{sec7}

In this final section we show how our squared linear models arise naturally in applications.
We are interested in discrete statistical models whose states are the
$k$-element subsets $\sigma$ of a  finite set $[n] = \{1, 2, . . . , n\}$,
and where the occurrence of elements in the subsets are negatively correlated.
In such scenarios, the model of choice is a  determinantal point process (DPP).
The DPP model is ubiquitous in
probability theory, statistical physics, algebraic combinatorics and machine learning \cite{BDF, KT, lyons, macchi}.

A general DPP may be written as the mixture of {\em projection DPPs}. 
The model parameter for a projection DPP is a $k \times n$ matrix $\Theta$,
and the probability of observing a particular $k$-subset $\sigma$ is proportional to ${\rm det}(\Theta_\sigma)^2$,
where $\Theta_\sigma$ is the $k \times k$-submatrix of $\Theta$ with column indices $\sigma$.
In algebraic statistics, one views the projection DPP as a projective variety in
$\PP^{\binom{n}{k}-1}$. This variety is the
{\em squared Grassmannian} ${\rm sGr}(k,n)$, whose coordinates
are the squares of the Pl\"ucker coordinates on the Grassmannian ${\rm Gr}(k,n)$.
For details and references see  \cite{DFRS, Fri}.

\begin{example}
    The braid arrangement in Example~\ref{ex:braid} may be realized as a projection DPP whose state space is $2$-subsets of $[n]$. The corresponding parameter matrix $\Theta$ is 
    \begin{align*}
        \begin{pmatrix}
            1 & 1 & \cdots & 1\\
            x_1 & x_2 & \cdots & x_n
        \end{pmatrix}.
    \end{align*}
\end{example}

We now consider the submodel given by 
a linear subspace $\mathcal{L}$ of the Pl\"ucker space $\PP^{\binom{n}{k}}-1$ which 
is contained in the Grassmannian ${\rm Gr}(k,n)$.
For instance, we can take   $\mathcal{L}$ to be the set of all
$k$-planes in $\R^n$
that contain a fixed $(k-1)$-plane.
Or, dually, we can consider all $k$-planes in $\R^n$ that lie in a fixed
$(k+1)$-plane. In terms of the parametrization above,
we take $\Theta$ to be a matrix whose
last row consists of unknowns $\theta_1,\theta_2,\ldots,\theta_n$
and whose other entries are fixed real numbers $a_{i,j}$.
We refer to this model as a  {\em linear projection DPP},
and we denote it by $\mathcal{L}^2 \subset {\rm sGr}(k,n)$.
We say that $\mathcal{L}^2$ is
 {\em generic} if the numbers $a_{i,j}$ are chosen~generically.

Using row operations, we can eliminate $k-1$ of the model parameters, and we can write
$$ \Theta \,\, = \,\,
\begin{pmatrix}
a_{1,1} & a_{1,2} & \cdots & a_{1,n-k+1}  & 1 & 0 & \cdots & 0 \, \,\\
a_{2,1} & a_{2,2} & \cdots & a_{2,n-k+1}  & 0 & 1 & \cdots & 0 \,\, \\
\vdots & \vdots & \ddots & \vdots & 0 & 0 & \ddots & 0\,\, \\
a_{k-1,1} & a_{k-1,2} & \cdots & a_{k-1,n-k+1}  & 0 & 0 & \cdots & 1\,\, \\
\theta_1 & \theta_2 & \cdots & \theta_{n-k+1} & 0 & 0 & \cdots  & 0\,\, \\
\end{pmatrix}.
$$
The probabilities in the model $\mathcal{L}^2$ are the squares
of all maximal minors of $\Theta$. Equivalently, the
probabilities are the squares of maximal minors
containing the last column in the matrix
\begin{equation}
\label{eq:LPDPP} \begin{pmatrix}
\,\, 1 & 0 & \cdots & 0 & - a_{1,1} & -a_{2,1} & \cdots & -a_{k-1,1} & \,\theta_1 \, \\
\,\, 0 & 1 & \cdots & 0 &- a_{1,2} & -a_{2,2} & \cdots & -a_{k-1,2} &  \,\theta_2 \, \\
\,\, \vdots & \vdots & \ddots & \vdots & \vdots & \vdots & \ddots        & \vdots  &  \,\vdots \,\\
\,\,  0 & 0 & \cdots & 1 & - a_{1,n-k+1} & - a_{2,n-k+1} & \cdots & -a_{k-1,n-k+1} &\,\theta_{n-k+1} \, \\
\end{pmatrix}.
\end{equation}
These $\binom{n}{k}$ minors are precisely the hyperplanes spanned by any $n-k$ of the first $n$ columns.

\begin{proposition} \label{prop:discriminantal}
A squared linear model is a linear projection DPP
if and only its hyperplane arrangement $\mathcal{A}$ is a {\em discriminantal arrangement}.
This means that the hyperplanes in $\mathcal{A}$ are those spanned by subsets of the 
$n$ points in $\PP^{n-k}$ seen in the first $n$ columns in (\ref{eq:LPDPP}).
\end{proposition}

Let us consider  generic models $\mathcal{L}^2$
where $\ell = n-k$ is fixed.
These are given by the $\binom{n}{\ell}$ hyperplanes which are
spanned by all $\ell$-subsets of $n$ generic points in $\PP^\ell$.
It is known that the number of regions of this discriminantal
arrangement is a polynomial in $n$ of degree $\ell^2$ \cite{ABFKST}.

\begin{corollary}
The ML degree of the generic linear projection DPP is
a polynomial in $n$ of degree $\ell^2$. For instance, if $\ell = 2$ then
the ML degree of the generic model $\mathcal{L}^2$ equals
\begin{equation}
\label{eq:quarticpoly}
 \frac{1}{8} (n-1)(n^3 - 5 n^2 + 14 n - 8). 
 \end{equation}
\end{corollary}
\begin{proof}
    The first sentence is a direct consequence of \cite[Theorem 3.3]{ABFKST}. 
    % The ML degree for $\ell = 2$ is obtained as $\frac 1 2 (\chi_\mathcal A(-1) + \chi_\mathcal A(1))$ using the expression \cite[p. 793]{KNT} for the characteristic polynomial $\chi_{\mathcal A}(t)$  at $t = -1$. 
    In \cite[p. 793]{KNT} the authors compute the characteristic polynomial $\chi_{\cA}(t)$ for the \emph{non-central} discriminantal arrangement. Then $\chi_{\cA}(1)$ is the number of bounded regions in $\R^n\backslash\cA$, while $\chi_{\cA}(-1)$ is the number of unbounded regions in $\R^n\backslash\cA$. To obtain the total number of regions in $\PP^{n-1}_{\R}\backslash\cA$, i.e., the ML degree, one then computes $\frac 1 2 (\chi_\mathcal A(-1) + \chi_\mathcal A(1))$. This is the polynomial \eqref{eq:quarticpoly}.
\end{proof}

\begin{example}[$n=6,k=4$]
Let $\Theta$ be a $4 \times 6$ matrix
whose first three rows are random real numbers
and whose last row consists of parameters
$\theta_1,\theta_2,\ldots,\theta_6$.
The probability of any quadruple in $[6]$ is the squared
determinant of the corresponding maximal minor of $\Theta$,
divided by the sum of all $15$ squared minors.
The variety $\mathcal{L}^2$ is a $3$-dimensional Veronese surface inside 
the $8$-dimensional squared Grassmannian ${\rm sGr}(4,6) \subseteq \PP^{14}$. 
The ML degree of this DPP model equals $70$. This
 the value of (\ref{eq:quarticpoly}) for $n= 6$.
 It counts the regions in the
 discriminantal arrangement for six points in $\PP^2$.
 These points are derived by the row operations on $\Theta$ that yield (\ref{eq:LPDPP}).
 All complex critical points are real, and
  there is one critical point in each region.
   \end{example} 
\section*{Acknowledgments}
We thank Benjamin Hollering and Yue Ren for helping
us with computations for Proposition \ref{prop:ideal_computations}
and Section \ref{sec5} respectively.

% \bibliographystyle{siamplain}
% \bibliography{references}

\end{document}